\documentclass[11pt]{amsart}

\usepackage[USenglish,english]{babel}
\usepackage{amsfonts}
\usepackage{amsmath}
\usepackage{amssymb}
\usepackage{amsthm}
\usepackage[T1]{fontenc}
\usepackage[latin1]{inputenc}
\usepackage{enumerate}
\usepackage{graphicx}
\usepackage[all]{xy}
\usepackage{accents}
\usepackage{mathtools}
\usepackage{mathrsfs}
\usepackage{comment}
\usepackage{accents}
\usepackage{tikz}
\usetikzlibrary{matrix,arrows,decorations.pathmorphing}
\usepackage{tikz-cd}
\usepackage{afterpage}
\usepackage{bbm}
\usepackage{dsfont}
\usepackage{stmaryrd}
\usepackage{hyperref}
\usepackage{mleftright}
\usepackage{subfig}
\usepackage{faktor}  

\theoremstyle{plain}
\newtheorem{thm}{Theorem}[section]
\newtheorem{lem}[thm]{Lemma}
\newtheorem{prop}[thm]{Proposition}
\newtheorem{cor}[thm]{Corollary}
\newtheorem*{thm*}{Theorem}
\newtheorem*{cor*}{Corollary}
\newtheorem{thmintro}{Theorem}

\newtheorem{corintro}[thmintro]{Corollary}
\newtheorem{propintro}[thmintro]{Proposition}

\theoremstyle{definition}
\newtheorem{defn}[thm]{Definition}
\newtheorem*{dfn*}{Definition}
\newtheorem{ex}[thm]{Example}
\newtheorem{rmk}[thm]{Remark}
\newtheorem*{rmk*}{Remarks}

\newtheorem*{quest*}{Question}

\renewcommand{\o}{\circ}
\newcommand{\wt}{\widetilde}

\newcommand{\R}{\mathbb{R}}
\newcommand{\Z}{\mathbb{Z}}
\newcommand{\N}{\mathbb{N}}
\newcommand{\Q}{\mathbb{Q}}

\renewcommand{\H}{\mathbb{H}}

\newcommand{\s}{\sigma}

\newcommand{\ra}{\rightarrow}
\newcommand{\cu}{\subseteq}

\newcommand{\G}{\Gamma}

\newcommand{\mbb}{\mathbb}
\newcommand{\mc}{\mathcal}
\newcommand{\mf}{\mathfrak}
\newcommand{\x}{\times}
\newcommand{\eps}{\epsilon}
\newcommand{\Om}{\Omega}
\newcommand{\om}{\omega}

\newcommand{\id}{\text{id}}

\newcommand{\acts}{\curvearrowright}
\newcommand{\wh}{\widehat}
\newcommand{\mscr}{\mathscr}

\begin{document}

\title{Superrigidity of actions on finite rank median spaces}
\author{Elia Fioravanti}

\begin{abstract}
Finite rank median spaces are a simultaneous generalisation of finite dimensional {\rm CAT}(0) cube complexes and real trees. If $\G$ is an irreducible lattice in a product of rank-one simple Lie groups, we show that every action of $\G$ on a complete, finite rank median space has a global fixed point. This is in sharp contrast with the behaviour of actions on infinite rank median spaces, where even proper cocompact actions can arise.

The fixed point property is obtained as corollary to a superrigidity result; the latter holds for irreducible lattices in arbitrary products of compactly generated topological groups.

We exploit \emph{Roller compactifications} of median spaces; these were introduced in \cite{Fioravanti1} and generalise a well-known construction in the case of cube complexes. We show that the Haagerup cocycle provides a reduced $1$-cohomology class that detects group actions with a finite orbit in the Roller compactification. This is new even for {\rm CAT}(0) cube complexes and has interesting consequences involving Shalom's property $H_{FD}$. For instance, in Gromov's density model, random groups at low density do not have $H_{FD}$. 
\end{abstract}

\maketitle

\tableofcontents

\section{Introduction.}

A \emph{median space} is a metric space $(X,d)$ such that, for any three points $x_1,x_2,x_3\in X$, there exists a unique point $m=m(x_1,x_2,x_3)\in X$ with the property that $d(x_i,m)+d(m,x_j)=d(x_i,x_j)$ for all $1\leq i<j\leq 3$. Simple examples are provided by real trees and by $\R^n$ with the $\ell^1$ metric. 

We remark that, to each connected median space $X$ of finite topological dimension, there corresponds a canonical {\rm CAT}(0) space $\wh{X}$, which is bi-Lip\-schitz equivalent to $X$ \cite{Bow4}. For instance, to $\R^n$ with the $\ell^1$ metric we associate $\R^n$ with its euclidean distance. 

More elaborate examples of median spaces are provided by Guirardel cores \cite{Guirardel} and by simply connected cube complexes satisfying Gromov's link condition \cite{HypGps}. In the latter case, we obtain a median space $X$ by endowing each cube with the $\ell^1$ metric, while $\wh{X}$ is the corresponding {\rm CAT}(0) cube complex. Median spaces generally display wilder features than cube complexes, as, like real trees, they can be essentially non-discrete objects. 

A clear illustration of this is the fact that, unlike cube complexes, the class of median spaces is closed under ultralimits. In particular, any asymptotic cone of any hierarchically hyperbolic group \cite{HHS} is described by a median space \cite{Behrstock-Drutu-Sapir2,Behrstock-Drutu-Sapir,Bow1,Zeidler}. Furthermore, we have observed in \cite{Fioravanti2} that general group actions on median spaces need not yield codimension-one subgroups; this is in contrast with a well-known key property of cube complexes \cite{Sageev,Gerasimov}.

These phenomena already appear among (connected) median spaces of finite \emph{rank}, i.e.\,those median spaces where the topological dimensions of locally compact subsets are uniformly bounded. Finite rank median spaces will be the main focus of the present paper and, indeed, all the examples listed above have finite rank.

Finite rank median spaces also retain many of the good combinatorial properties of cube complexes. In addition to the existence of a {\rm CAT}(0) metric, their horoboundaries are compatible with the median property \cite{Fioravanti1} and many groups of isometries contain free non-abelian subgroups \cite{Fioravanti2}. We would expect many known results for {\rm CAT}(0) cube complexes to extend to finite rank median spaces without significant complications, for instance \cite{BCGNW,Guentner-Higson,Nevo-Sageev,CFI,Kar-Sageev,Fernos,Fernos-Lecureux-Matheus} to name a few. 

All the above similarities between cube complexes and more general median spaces can be ascribed to the existence of a collection $\mscr{W}$ of \emph{walls}. These encode the geometry of the space in the same way as hyperplanes do in {\rm CAT}(0) cube complexes. The set $\mscr{W}$ should not be thought of as discrete and, in fact, it needs to be endowed with a measure $\mu$ that expresses the ``thickness'' of sets of walls \cite{CDH,Fioravanti1}. Indeed, the concept of median space is dual to the notion of \emph{space with measured walls} \cite{Cherix-Martin-Valette,Cornulier-Tessera-Valette,CDH} --- similar to \emph{spaces with walls} \cite{Haglund-Paulin} being the dual viewpoint on {\rm CAT}(0) cube complexes \cite{Sageev,Nica,Chatterji-Niblo}. 

Our main theorem is a superrigidity result for irreducible lattices $\G$ in products $G=G_1\x ... \x G_{\ell}$ of locally compact, compactly generated groups. Namely, under weak assumptions of non-elementarity, every action of $\G$ on a finite rank median space $X$ essentially arises from \emph{continuous} isometric actions ${G_i\acts Y_i}$ on median spaces of lower rank. For cube complexes, this was known due to \cite{CFI}.

In the more general context of {\rm CAT}(0) spaces, similar results were obtained long ago in \cite{Monod,Caprace-Monod(A)}. Unfortunately, applying these to a median space $X$ only provides actions of the factors $G_i$ on {\rm CAT}(0) subspaces $Z_i\cu\wh{X}$; these subspaces might bear no relation to the median structure on $X$. This might seem like an irrelevant subtlety, but it is, on the contrary, key to the fixed point properties that this paper provides.

As an illustration of this, consider a uniform irreducible lattice $\G$ in the product $SL_2(\R)\x SL_2(\R)$. It acts properly and cocompactly on the {\rm CAT}(0) space $\H^2\x\H^2$. In particular, it is quasi-isometric to a product of surface groups, hence coarse median of finite rank (in the sense of \cite{Bow1}). It was moreover shown in \cite{Chatterji-Drutu} that $\G$ acts properly and \emph{cocompactly} on a median space of infinite rank. It should thus appear particularly striking that every action of $\G$ on a complete, connected, \emph{finite rank} median space fixes a point; this follows from our results, see Corollary~\ref{D} below. 

The proof of our superrigidity theorem follows a very similar outline to Monod's \cite{Monod}. This is mostly hidden in our application of Shalom's superrigidity (Theorem~4.1 in \cite{Shalom}), thus we believe it is important to highlight the analogy here. We have already mentioned that, to each finite rank median space $X$, one can associate a \emph{finite dimensional} {\rm CAT}(0) space $\wh{X}$. However, one can also consider the \emph{infinite dimensional} {\rm CAT}(0) (in fact, Hilbert) space $L^2(\mscr{W},\mu)$.

Retracing the proof of Monod's superrigidity in our context, we would start with an action $\G\acts X$ and induce a continuous action of $G$ on the space $L^2(G/\G,\wh{X})$ (see \cite{Monod} for a definition). Then, we would prove that a subspace of $L^2(G/\G,\wh{X})$ splits as a product $Z_1\x ... \x Z_k$, where the action of $G$ on the $i$-th factor only depends on the projection to $G_i$. Finally, we would carry back to $\G\acts\wh{X}$ the information gained about $\G\acts L^2(G/\G,\wh{X})$.

Our application of Shalom's machinery instead constructs a continuous action of $G$ on the space $L^2(G/\G,L^2(\mscr{W},\mu))$. Again, one then proves a splitting theorem for this action and carries the gained insight back to the action $\G\acts L^2(\mscr{W},\mu)$. Our main contribution lies in transferring information back and forth between the actions ${\G\acts L^2(\mscr{W},\mu)}$ and $\G\acts X$. In particular, Shalom's machinery can only be set in motion once we have a nonvanishing reduced cohomology class for $\G\acts L^2(\mscr{W},\mu)$. We provide this by a careful study of the Haagerup cocycle, which is entirely new even in the context of  {\rm CAT}(0) cube complexes and has many interesting consequences regarding Shalom's property $H_{FD}$ (see Theorem~\ref{Z} and Corollaries~\ref{G} and~\ref{H} below).

We now describe our results in greater detail.

\medskip
\noindent
1.1. {\bf A cohomological characterisation of elementary actions.}
Each median space $X$ has a distinguished collection $\mscr{H}$ of subsets called \emph{halfspaces}; every wall gives rise to two halfspaces and each halfspace arises from a wall. The collection $\mscr{H}$ is equipped with a measure $\wh{\nu}$, see \cite{Fioravanti1}. In the case of cube complexes, one recovers the usual notion of halfspace and $\wh{\nu}$ is simply the counting measure.

Given a topological group $G$ and an isometric action $G\acts X$, one naturally obtains a unitary representation $\rho\colon G\ra\mscr{U}(L^2(\mscr{H},\wh{\nu}))$ and a cocycle $b\colon G\ra L^2(\mscr{H},\wh{\nu})$ --- the \emph{Haagerup cocycle}. This construction is well-known and appears for instance in \cite{Cherix-Martin-Valette, Cornulier-Tessera-Valette, CDH, Fernos-Valette}. If $G\acts X$ has continuous orbits, $\rho$ and $b$ are continuous; thus $b$ induces a reduced continuous cohomology class $\overline{[b]}\in\overline{H^1_c}(G,\rho)$.

In \cite{Fioravanti2}, we introduced a notion of elementarity for actions on median spaces; namely, we say that $G\acts X$ is \emph{Roller elementary} if $G$ has at least one finite orbit within the Busemann compactification $\overline X$; the latter is also known as \emph{Roller compactification}. Roller elementarity implies --- but is in general strictly stronger than --- the existence of a finite orbit in the visual compactification of the {\rm CAT}(0) space $\wh{X}$. 

If $G\acts X$ is an isometric action with continuous orbits, Roller elementarity can be described in terms of the Haagerup class $\overline{[b]}$.

\begin{thmintro}\label{A}
Let $X$ be a complete, finite rank median space. The Haagerup class $\overline{[b]}\in\overline{H^1_c}(G,\rho)$ vanishes if and only if $G\acts X$ is Roller elementary.
\end{thmintro}

Theorem~\ref{A} extends various known results. In the case of simplicial trees, it appears in \cite{Fernos-Valette}. For {\rm CAT}(0) cube complexes, the implication ``\emph{Roller nonelementary $\Rightarrow$ $\overline{[b]}\neq 0$}'' is implicit in \cite{Delzant-Py}. In \cite{CFI}, the authors construct a family of bounded cohomology classes detecting Roller elementarity in {\rm CAT}(0) cube complexes. 

We remark that Theorem~\ref{A} equally holds if we replace $L^2(\mscr{H},\wh{\nu})$ with any $L^p(\mscr{H},\wh{\nu})$, $1\leq p<+\infty$, although it is slightly simpler to exploit the richer structure of Hilbert spaces in its proof.

Our superrigidity result only relies on the implication of Theorem~\ref{A} that yields $\overline{[b]}\neq 0$, but we believe the full statement of Theorem~\ref{A} to be of independent interest. The proof of the other implication turns out to be quite technical and requires a careful study of the structure of \emph{unidirectional boundary sets (UBS's)} in median spaces; these are a generalisation of the simplices in Hagen's simplicial boundary of a {\rm CAT}(0) cube complex \cite{Hagen}. Most of these details will be relegated to the appendix.
 
\medskip
\noindent
1.2. {\bf Superrigidity of actions.}
Once we have a nontrivial reduced cohomology class, as provided by Theorem~\ref{A}, we can apply well-established machinery (namely Theorem~4.1 in \cite{Shalom}) to obtain superrigidity results.

Let $X$ be a complete, finite rank median space. Its Roller compactification $\overline X$ is partitioned into \emph{components} \cite{Fioravanti1}. The subset $X\cu\overline X$ forms a whole component and every other component is itself a complete median space of strictly lower rank. This aspect of $\overline X$ shares many similarities with refined boundaries of {\rm CAT}(0) spaces \cite{Leeb-fine,Caprace-fine} and with Satake compactifications of symmetric spaces \cite{Borel-Ji,Borel-Ji2}.

Given a component $Z\cu\overline X$, a \emph{median subalgebra} of $Z$ is a subset $Y\cu Z$ that is itself a median space with the restriction of the median metric of $Z$; equivalently, the median map $m\colon Z^3\ra Z$ takes $Y^3$ into $Y$. We are now ready to state our main superrigidity result.

\begin{thmintro}\label{B}
Let $X$ be a complete, finite rank median space. Let $\G$ be a uniform, irreducible lattice in a product ${G=G_1\x ... \x G_{\ell}}$ of compactly generated, locally compact groups with $\ell\geq 2$. Suppose $\G\acts X$ is a Roller non\-e\-le\-men\-ta\-ry action. There exist a finite index subgroup $\G_0\leq\G$, a $\G_0$-invariant component $Z\cu\overline X$ and a $\G_0$-invariant closed median subalgebra $Y\cu Z$ where the action $\G_0\acts Y$ extends to a continuous action $G_0\acts Y$, for some open finite index subgroup $G_0\leq G$.
\end{thmintro}

\begin{rmk*}
\begin{enumerate}
\item Theorem~\ref{B} also applies to nonuniform lattices, as long as they are \emph{square-integrable}; this is a well-known technical condition that implies finite generation and ensures that Theorem~4.1 in \cite{Shalom} still holds. 

All irreducible lattices in $G_1\x ...\x G_{\ell}$ are square-integrable if each $G_i$ is the group of $k_i$-rational points of a semisimple, almost $k_i$-simple, $k_i$-isotropic linear algebraic group defined over some local field $k_i$ \cite{Shalom}. Further examples of nonuniform square-in\-te\-grable lattices include minimal Kac-Moody groups over sufficiently large finite ground fields; these can be regarded as irreducible lattices in the product of the closed automorphism groups of the associated buildings \cite{Remy99,Remy05}.

\item Theorem~\ref{B} should be compared to Shalom's superrigidity result for actions on simplicial trees, Theorem~0.7 in \cite{Shalom}. If $X$ is a simplicial tree, $Y$ is always a subcomplex of $X$ and $\G_0=\G$, $G_0=G$. The complications in the statement of Theorem~\ref{B} reflect phenomena that do not happen in the world of trees. 

However, as soon as we leave the context of rank-one median spaces (i.e.\,real trees), our result is optimal even if one restricts to {\rm CAT}(0) square complexes; see Examples~\ref{need Y} and~\ref{need G_0}. We remark that, when $X$ is a general {\rm CAT}(0) square complex, the median algebra $Y$ might not be a \emph{subcomplex} of $X$ or $Z$.

\item We can take $\G_0=\G$, $G_0=G$ and $Z=X$ as long as $X$ is irreducible, $G$ has no finite orbits in the visual compactification of $\wh{X}$ and $G$ leaves invariant no proper closed convex subset of $X$; see Theorem~\ref{SR 1} below. However, even in this case, the action in general extends only to a \emph{proper} median subalgebra of $X$.

\item For {\rm CAT}(0) cube complexes, the superrigidity result of \cite{CFI} is slightly more general than Theorem~\ref{B} as it applies to \emph{all} nonuniform lattices. This is due to their use of bounded cohomology (namely Theorem~16 in \cite{Burger-Monod}), rather than reduced cohomology. In the setting of {\rm CAT}(0) cube complexes, our strategy of proof was instead hinted at on page~9 of \cite{Shalom}.
\end{enumerate}
\end{rmk*}

\medskip
\noindent
1.3. {\bf Fixed point properties for lattices.}
Unlike automorphism groups of {\rm CAT}(0) cube complexes, the isometry group of a median space needs not be totally disconnected. Still, it is possible to exploit Theorem~\ref{B} to derive a fixed point property for irreducible lattices in connected groups.

Given a locally compact topological group $Q$, we denote the connected component of the identity by $Q^0$. We say that $Q$ satisfies \emph{condition $(\ast)$} if $Q/Q^0$ is amenable or has Shalom's property $H_{FD}$ (see Section~1.5 for a definition). In particular, all almost-connected and connected-by-(T) groups satisfy condition $(\ast)$. 

\begin{thmintro}\label{C}
Let $X$ be a complete, finite rank median space. Let $\G$ be a square-integrable, irreducible lattice in a product $G_1\x ... \x G_{\ell}$ with $\ell\geq 2$. Suppose that every $G_i$ is compactly generated and satisfies condition $(*)$. 
\begin{enumerate}
\item Every action $\G\acts X$ is Roller elementary. 
\item If $\G$ does not virtually map onto $\Z$, every action $\G\acts X$ has a finite orbit within $X$. If moreover $X$ is connected, every action $\G\acts X$ has a global fixed point.
\end{enumerate}
\end{thmintro}

When $X$ is a real tree, Theorem~\ref{C} also follows from Theorem~6 in \cite{Monod}. We remark that, contrary to part~2 of Theorem~\ref{C}, every group that virtually maps onto $\Z$ admits a Roller elementary action on $\R^n$ with unbounded orbits.

\begin{corintro}\label{D}
Let $X$ be a complete, connected, finite rank median space. Let $\G$ be any irreducible lattice in a connected, higher rank, semisimple Lie group $G$. Every action $\G\acts X$ fixes a point.
\end{corintro}

The assumption that $X$ have finite rank is essential for Corollary~\ref{D} to hold. If at least one simple factor $G_i<G$ is locally isomorphic to $O(n,1)$ or $U(n,1)$, $n\geq 2$, the lattice $\G$ admits an action on an infinite rank median space with unbounded orbits \cite{CDH}. On the other hand, if each simple factor of $G$ has rank at least two, then $\G$ has property (T) and Corollary~\ref{D} follows from Theorem~1.2 in \cite{CDH}.  

An analogue of Corollary~\ref{D} for {\rm CAT}(0) cube complexes was proved in \cite{CFI}. We remark however that the full strength of Corollary~\ref{D} reveals a surprising new phenomenon. If all $G_i$'s are locally isomorphic to $O(n,1)$ and $\G$ is uniform, it was shown in \cite{Chatterji-Drutu} that $\G$ admits a proper and \emph{cocompact} action on an infinite rank median space. Corollary~\ref{D} shows such actions cannot be reduced to actions on finite rank median spaces. By contrast, any cocompact {\rm CAT}(0) cube complex is finite dimensional.

\medskip
\noindent
1.4. {\bf Homomorphisms to coarse median groups.} 
Coarse median spaces were introduced in \cite{Bow1} as an attempt to formulate a coarse notion of nonpositive curvature. They have recently received a lot of attention \cite{Bow6,Haettel1,Zeidler,Spakula-Wright,NWZ,NWZ2,ANWZ} and proved instrumental to striking results such as \cite{Haettel2,HHS3}. 

A finitely generated group is said to be \emph{coarse median} if its Cayley graphs are coarse median spaces; this property is independent of the chosen finite generating set. Examples of coarse median groups of finite rank include hyperbolic groups, mapping class groups, uniform lattices in products of rank-one simple Lie groups, fundamental groups of closed irreducible 3-manifolds not modelled on Nil or Sol and all cubulated groups --- in particular, right-angled Coxeter and Artin groups. More generally, if a group is a \emph{hierarchically hyperbolic space (HHS)}, it is coarse median of finite rank \cite{HHS,HHS2,Bow3}.

We will be mainly interested in \emph{equivariantly coarse median} groups. If we view coarse median groups as a generalisation of groups that are HHS, \emph{equivariantly} coarse median groups generalise \emph{hierarchically hyperbolic groups (HHG)}. In particular, all the examples above are also equivariantly coarse median of finite rank (with the exception of some graph manifold groups).

More precisely, we say that a group $H$ is \emph{equivariantly coarse median} if it is equipped with a finite generating set $S\cu H$ and a coarse median ${\mu\colon H^3\ra H}$ \cite{Bow1} such that $d_S(\mu(ha,hb,hc),h\mu(a,b,c))\leq C<+\infty$ for all elements $h,a,b,c\in H$; here $d_S$ denotes the word metric induced by $S$. Note that this definition does not depend on the choice of $S$. Equivariantly coarse median groups have already been considered in \cite{Zeidler} under the name of ``uniformly left-invariant coarse median structures''. 

If $H$ is a coarse median group of finite rank, every asymptotic cone of $H$ is endowed with a canonical, bi-Lipschitz equivalent median metric \cite{Zeidler}. As an example, in asymptotic cones of mapping class groups the median geodesics are limits of hierarchy paths \cite{Behrstock-Drutu-Sapir2,Behrstock-Drutu-Sapir}. When $H$ is \emph{equivariantly} coarse median, the median metric on each asymptotic cone is preserved by the action of the ultrapower of $H$.

Given a group $\G$ and an infinite sequence of pairwise non-conjugate homomorphisms $\G\ra H$, we can apply the Bestvina-Paulin construction \cite{Bestvina,Paulin}. The result is an isometric action on a median space $X$ with unbounded orbits; this is obtained as the canonical median space bi-Lipschitz equivalent to an asymptotic cone of $H$. Along with Theorem~\ref{C}, this implies:

\begin{corintro}\label{E}
Let $H$ be an equivariantly coarse median group of finite rank. Let $\G$ be as in the second part of Theorem~\ref{C}. There exist only finitely many pairwise non-conjugate homomorphisms $\G\ra H$. 
\end{corintro}

When $\G$ is a uniform irreducible lattice in a product $G$ of rank-one simple Lie groups, we can take $\G=H$ and Corollary~\ref{E} provides an alternative proof that $\mathrm{Out}(\G)$ is finite, without relying on Margulis' superrigidity, nor on the structure of the normaliser of $\G$ in $G$ (Lemma II.6.3 in \cite{Margulis}).

If instead $\G$ is a lattice in a product of higher-rank, simple Lie groups, compare Theorem~5.6 in \cite{Zeidler} and Corollary~D in \cite{Haettel2}. Also see Theorem~4.1 in \cite{Bader-Furman}.  

We remark that Corollary~\ref{E} can be strengthened significantly if $H$ acts freely on a complete, finite rank median space. Indeed, the following is an immediate consequence of Theorem~F in \cite{Fioravanti2}.

\begin{propintro}\label{F}
Let $H$ be a group admitting a free action on a complete, finite rank median space $X$. Suppose that every action $\G\acts X$ is Roller elementary. Then every homomorphism $\G\ra H$ factors through a virtually abelian subgroup of $H$.
\end{propintro}

Proposition~\ref{F} applies for instance to the case when $\G$ has no non-abelian free subgroups, has property $H_{FD}$ (see Theorem~\ref{Z} below) or satisfies the hypotheses of the first part of Theorem~\ref{C}. In particular, if $\G$ is an irreducible lattice in a connected, higher rank, semisimple Lie group, every homomorphism $\G\ra H$ has finite image.

This should motivate a certain interest in groups acting freely on complete, finite rank median spaces. If a group acts freely on a finite dimensional {\rm CAT}(0) cube complex, it clearly falls into this class; however, it is unclear at this stage whether these are the only finitely generated examples. See \cite{CasalsRuiz-Kazachkov} for partial results in this direction.

Note that the \emph{infinitely generated} group $\Q$ admits a proper action on a rank two median space, namely the product of a simplicial tree and the real line (Example~II.7.13 in \cite{BH}). However, since $\Q$ is a divisible group, all its elements must act elliptically on any (possibly infinite-dimensional) {\rm CAT}(0) cube complex \cite{Haglund}.

Even within finitely generated groups, actions on median spaces tend to be more flexible than actions on {\rm CAT}(0) cube complexes. For every group $H$, we can consider $\dim_{fm}H$, i.e.\,the minimum rank of a complete {\bf m}edian space $X$ admitting a {\bf f}ree action of $H$; if $H$ does not act freely on any complete median space, we set $\dim_{fm}H:=-1$. Restricting to {\rm CAT}(0) {\bf c}ube complexes, we can similarly define $\dim_{fc}H$ and, if we only consider (metrically) {\bf p}roper actions, we obtain $\dim_{pm}H$ and $\dim_{pc}H$. Thus, $\dim_{pc}\Q=\dim_{fc}\Q=-1$, while $\dim_{pm}\Q=2$ and $\dim_{fm}\Q=1$ (the latter since $\Q<\R$).

We remark that $1\leq\dim_{fm}H<\dim_{cm}H$ and $1\leq\dim_{pm}H<\dim_{pm}H$ for many finitely generated groups $H$. For instance, $\dim_{fc}H=1$ if and only if $H$ is free. On the other hand, by work of E. Rips, $\dim_{fm}H=1$ if and only if $H$ is a free product of free abelian and surface groups (excluding a few nonorientable surfaces); see e.g.\,Theorem~9.8 in \cite{Bestvina-Feighn}. One can use the same observation to construct free actions of various right-angled Artin groups on median spaces of rank strictly lower than the dimension of the Salvetti complex.

Considering more general actions, we mention that there exist finitely generated groups admitting actions on real trees with unbounded orbits, but whose actions on finite dimensional {\rm CAT}(0) cube complexes must all have global fixed points. An example is provided by the group $L$ in Section~2 of \cite{Minasyan}; the fixed point property for actions on finite dimensional cube complexes can easily be derived from the same property of Thompson's group $V$ \cite{Gen-V,Motoko}.

\medskip
\noindent
1.5. {\bf Shalom's property $H_{FD}$ and random groups.}
Theorem~\ref{A} also allows us to prove that various (non-amenable) groups do not have property $H_{FD}$. The latter was introduced in \cite{Shalom2}: a topological group $G$ has \emph{property $H_{FD}$} if every unitary representation $\pi$ with $\overline{H^1_c}(G,\pi)\neq 0$ has a finite dimensional subrepresentation. 

Property $H_{FD}$ is trivially satisfied by every locally compact group with property (T) \cite{Delorme}, but also, at the opposite end of the universe of groups, by a large class of amenable groups. This includes polycyclic groups, lamplighter groups and all connected, locally compact, a\-me\-nable groups \cite{Shalom2, Martin, Erschler-Ozawa}. An example of an amenable group without $H_{FD}$ is provided by the wreath product $\Z\wr\Z$ \cite{Shalom2}. We prove the following:

\begin{thmintro}\label{Z}
If $G$ has property $H_{FD}$, every isometric action of $G$ on a complete, finite rank median space is Roller elementary. 
\end{thmintro}

\begin{corintro}\label{G}
Let $\G$ be a discrete group with property $H_{FD}$. If $\G$ acts freely and cocompactly on a {\rm CAT}(0) cube complex $X$, then $\G$ is virtually abelian.
\end{corintro}

Property $H_{FD}$ has been studied almost exclusively within the class of amenable groups, where it happens to be a quasi-isometry invariant \cite{Shalom2}. It was a key ingredient (implicitly, or explicitly) in recent more elementary proofs of Gromov's theorem on groups of polynomial growth \cite{Kleiner, Ozawa}. It has moreover interesting applications to the study of quasi-isometric embeddings into Hilbert spaces \cite{CTV}.

Property $H_{FD}$ is inherited by uniform lattices and is stable under direct products and central extensions \cite{Shalom2}. Being satisfied by groups that fall into two extremely different classes, namely amenable and Kazhdan groups, it is reasonable to expect a wide variety of groups with property $H_{FD}$. However, it seems that no answer is known to the following question. 

\begin{quest*}
Does every finitely generated group with property $H_{FD}$ virtually split as a direct product of an amenable group and finitely many groups with property (T)? Does every word hyperbolic group with property $H_{FD}$ also satisfy property (T)?
\end{quest*}

Corollary~\ref{G} and the results of \cite{Ollivier-Wise} imply that random groups at low density do not satisfy $H_{FD}$.

\begin{corintro}\label{H}
With overwhelming probability, random groups at density ${d<\frac{1}{6}}$ in Gromov's density model do not have property $H_{FD}$.
\end{corintro}

Note however that, at density $d>\frac{1}{3}$, random groups are Kazhdan \cite{Zuk, Kotowski-Kotowski}, hence satisfy property $H_{FD}$. 

\medskip
\noindent
{\bf Acknowledgements.} The author warmly thanks Brian Bowditch, Pier\-re-Em\-ma\-nuel Caprace, Indira Chatterji, Yves Cornulier, Thomas Delzant, Mark Hagen, Masato Mimura, Narutaka O\-za\-wa, Romain Tessera, Pierre Pansu, Alain Valette for helpful conversations. The author expresses special gratitude to Cornelia Dru\c tu and Talia Fern\'os for their interest and encouragement throughout the writing of this paper. The author also wishes to thank the anonymous referee for many valuable comments.

This work was undertaken at the Mathematical Sciences Research Institute in Berkeley during the Fall 2016 program in Geometric Group Theory, where the author was supported by the National Science Foundation under Grant no.\,DMS-1440140 and by the GEAR Network. Part of this work was also carried out at the Isaac Newton Institute for Mathematical Sciences, Cambridge, during the programme ``Non-positive curvature, group actions and cohomology'' and was supported by EPSRC grant no.\,EP/K032208/1. The author was also supported by the Clarendon Fund and the Merton Moussouris Scholarship.

\section{Preliminaries.}

\subsection{Median spaces and median algebras.}\label{prelims}

Let $X$ be a metric space. Given points $x,y\in X$, the \emph{interval} $I(x,y)$ is the set of points $z\in X$ that lie between $x$ and $y$, i.e.\,that satisfy $d(x,y)=d(x,z)+d(z,y)$. We say that $X$ is a \emph{median space} if for all $x,y,z\in X$ there exists a unique point $m(x,y,z)$ that lies in $I(x,y)\cap I(y,z)\cap I(z,x)$. The \emph{median map} $m\colon X^3\ra X$ that we obtain this way endows $X$ with a structure of \emph{median algebra} (see \textsection1.1 in \cite{Roller}). Most definitions in the theory of median spaces can also be given for arbitrary median algebras; we will follow this approach in introducing the necessary notions. The reader can consult e.g.\,\cite{Roller,Nica-thesis,CDH,Bow1,Bow4,Fioravanti1,Fioravanti2} for more background on median spaces and algebras. 

In a median space, $I(x,y)=\{z\in I(x,y)\mid z=m(x,y,z)\}$; this can be taken as a definition of intervals in general median algebras. If $(M,m)$ is a median algebra, we say that a subset $C\cu M$ is \emph{convex} if $I(x,y)\cu C$ whenever $x,y\in C$. The intersection of a finite family of pairwise intersecting convex sets is always nonempty; this is known as Helly's Theorem, see Theorem~2.2 in \cite{Roller}.

A subset $\mf{h}\cu M$ is a \emph{halfspace} if both $\mf{h}$ and $\mf{h}^*:=M\setminus\mf{h}$ are convex; we will denote the set of halfspaces of $M$ by $\mscr{H}(M)$, or simply by $\mscr{H}$ when there is no ambiguity. Halfspaces $\mf{h},\mf{k}$ are said to be \emph{transverse} if no two distinct elements of the set $\{\mf{h},\mf{h}^*,\mf{k},\mf{k}^*\}$ are comparable in the poset $(\mscr{H},\cu)$. Equivalently, the intersections $\mf{h}\cap\mf{k}$, $\mf{h}\cap\mf{k}^*$, $\mf{h}^*\cap\mf{k}$, $\mf{h}^*\cap\mf{k}^*$ are all nonempty.

Given $A\cu\mscr{H}$, we write $A^*$ for $\{\mf{h}\in\mscr{H}\mid\mf{h}^*\in A\}$. A subset $\s\cu\mscr{H}$ is said to be an \emph{ultrafilter} if any two halfspaces in $\s$ intersect and $\mscr{H}=\s\sqcup\s^*$. For instance, for each $x\in M$ the set $\s_x:=\{\mf{h}\in\mscr{H}\mid x\in\mf{h}\}$ is an ultrafilter. 

Given subsets $A,B\cu M$, we write ${\mscr{H}(A|B):=\{\mf{h}\in\mscr{H}\mid B\cu\mf{h},~A\cu\mf{h}^*\}}$ and $\s_A:=\mscr{H}(\emptyset|A)$; we refer to sets of the form $\mscr{H}(x|y)$, $x,y\in M$, as \emph{halfspace intervals}. If $C,C'\cu M$ are disjoint and convex, the set $\mscr{H}(C|C')$ is nonempty, see Theorem~2.7 in \cite{Roller}. In particular $\s_x=\s_y$ if and only if the points $x,y\in M$ coincide. 

A subset $\Om\cu\mscr{H}$ is \emph{inseparable} if, whenever $\mf{j}\in\mscr{H}$ satisfies $\mf{h}\cu\mf{j}\cu\mf{k}$ for $\mf{h},\mf{k}\in\Om$, we have $\mf{j}\in\Om$. Given a subset $A\cu\mscr{H}$, its \emph{inseparable closure} is the smallest inseparable subset of $\mscr{H}$ that contains $A$; it coincides with the union of the sets $\mscr{H}(\mf{k}^*|\mf{h})$, for $\mf{h},\mf{k}\in A$.

A \emph{wall} is a set of the form $\mf{w}=\{\mf{h},\mf{h}^*\}$, with $\mf{h}\in\mscr{H}$; we say that $\mf{h}$ and $\mf{h}^*$ are the \emph{sides} of $\mf{w}$. The wall $\mf{w}$ \emph{separates} subsets $A,B\cu M$ if either $\mf{h}$ or $\mf{h}^*$ lies in $\mscr{H}(A|B)$; we denote by $\mscr{W}(A|B)=\mscr{W}(B|A)$ the set of walls separating $A$ and $B$ and by $\mscr{W}(M)$, or simply $\mscr{W}$, the set of all walls of the median algebra $M$. A wall is \emph{contained} in a halfspace $\mf{k}$ if one of its sides is; a wall $\mf{w}$ is contained in disjoint halfspaces $\mf{k}_1,\mf{k}_2$ if and only if $\mf{k}_2=\mf{k}_1^*$ and $\mf{w}=\{\mf{k}_1,\mf{k}_2\}$. If a side of the wall $\mf{w}_1$ is transverse to a side of the wall $\mf{w}_2$, we say that $\mf{w}_1$ and $\mf{w}_2$ are \emph{transverse}. 

The \emph{rank} of the median algebra $M$ is the maximum cardinality of a set of pairwise transverse walls; various alternative (and equivalent) definitions of the rank can be found in Proposition~6.2 of \cite{Bow1}. We remark that $M$ has rank zero if and only if it consists of a single point. 

If $X$ is a connected median space, its rank coincides with the supremum of the topological dimensions of its locally compact subsets --- even when either of the two quantities is infinite. See Theorem~2.2 and Lem\-ma~7.6 in \cite{Bow1} for one inequality and Proposition~5.6 in \cite{Bow4} for the other. 

When $X$ is complete, connected and finite rank, $X$ is bi-Lipschitz equivalent to a canonical {\rm CAT}(0) space $\wh{X}$ \cite{Bow4}. The visual boundary of $\wh{X}$ is finite dimensional by Proposition~2.1 in \cite{CL}. Every isometry of $X$ extends to an isometry of $\wh{X}$ yielding a homomorphism $\text{Isom}~X\hookrightarrow\text{Isom}~\wh{X}$. Every convex subset of $X$ is also convex in $\wh{X}$; the converse is not true: the euclidean convex hull of the points $(1,0,0)$, $(0,1,0)$ and $(1,1,1)$ in the cube $[0,1]^3$ is not even a median subalgebra.  

Halfspaces in finite rank median spaces are fairly well-behaved. See Corollary~2.23 and Proposition~2.26 in \cite{Fioravanti1} for a proof of the following:

\begin{prop}\label{trivial chains of halfspaces}
Let $X$ be a complete median space of finite rank $r$. Every halfspace is either open or closed (possibly both). Moreover, if $\mf{h}_1\supsetneq ... \supsetneq\mf{h}_k$ is a chain of halfspaces with $\overline{\mf{h}_1^*}\cap\overline{\mf{h}_k}\neq\emptyset$, we have $k\leq 2r$.
\end{prop}

The following is a simple but extremely useful observation: given ultrafilters ${\s_1,\s_2\cu\mscr{H}(M)}$ and $\mf{h},\mf{k}\in\s_1\setminus\s_2$, we either have $\mf{h}\cu\mf{k}$, or $\mf{k}\cu\mf{h}$, or $\mf{h}$ and $\mf{k}$ are transverse. Along with Dilworth's Theorem \cite{Dilworth} this yields the following.

\begin{lem}\label{Dilworth for differences}
Let $M$ be a median algebra of finite rank $r$ and let $\s_1,\s_2\cu\mscr{H}$ be ultrafilters. 
\begin{enumerate}
\item We can decompose $\s_1\setminus\s_2=\mc{C}_1\sqcup ... \sqcup\mc{C}_k$, where $k\leq r$ and each $\mc{C}_i$ is nonempty and totally ordered by inclusion.
\item Every infinite subset of $\s_1\setminus\s_2$ contains an infinite subset that is totally ordered by inclusion.
\end{enumerate}
\end{lem}

If $C\cu M$ is a subset and $x\in M$, a \emph{gate} for $(x,C)$ is a point $y\in C$ such that $y\in I(x,z)$ for every $z\in C$; gates are unique when they exist. If a gate exists for every point of $M$, we say that $C$ is \emph{gate-convex}; in this case we can define a \emph{gate-projection} $\pi_C\colon M\ra C$ by associating to each point of $M$ the unique gate. The gate-projection to $C$ is a morphism of median algebras and satisfies $\mscr{W}(x|C)=\mscr{W}(x|\pi_C(x))$ for every $x\in M$; see Proposition~2.1 and Lemma~2.4 in \cite{Fioravanti1}.

Gate-convex subsets are always convex, but the converse is not always true. For every $y,z\in M$, the interval $I(y,z)$ is gate-convex with gate-projection $x\mapsto m(x,y,z)$. The following is a summary of Lemma~2.2, Proposition~2.3 and Lemma~2.4 in \cite{Fioravanti1}.

\begin{prop}\label{all about gates}
Let $C,C'\cu M$ be gate-convex.
\begin{enumerate}
\item The sets $\{\mf{h}\in\mscr{H}(M)\mid\mf{h}\cap C\neq\emptyset,~\mf{h}^*\cap C\neq\emptyset\}$, ${\{\pi_C^{-1}(\mf{h})\mid\mf{h}\in\mscr{H}(C)\}}$ and $\mscr{H}(C)$ are all naturally in bijection.
\item There exists a \emph{pair of gates}, i.e.\,a pair $(x,x')$ of points $x\in C$ and $x'\in C'$ such that $\pi_C(x')=x$ and $\pi_{C'}(x)=x'$. In particular, we have ${\mscr{H}(x|x')=\mscr{H}(C|C')}$.
\item The set $\pi_C(C')$ is gate-convex with gate-projection $\pi_C\o\pi_{C'}$. Moreover, $\pi_C\o\pi_{C'}\o\pi_C=\pi_{C}\o\pi_{C'}$.
\item If $C\cap C'\neq\emptyset$, we have $\pi_C(C')=C\cap C'$ and $\pi_C\o\pi_{C'}=\pi_{C'}\o\pi_C$. In particular, if $C'\cu C$, we have $\pi_{C'}=\pi_{C'}\o\pi_C$.
\end{enumerate}
\end{prop}

A median algebra $(M,m)$ endowed with a Hausdorff topology is said to be a \emph{topological median algebra} if the median map $m\colon M^3\ra M$ is continuous; here we equip $M^3$ with the product topology. Median spaces always provide topological median algebras; indeed, the median map $m$ is $1$-Lipschitz in that case (see Corollary~2.15 in \cite{CDH}). 

In compact median algebras and complete median spaces, a subset is gate-convex if and only if it is closed and convex; moreover, gate-projections are continuous. In median spaces, gate-projections are even $1$-Lipschitz. See Lemmas~2.6 and~2.7 in \cite{Fioravanti1} and Lemma~2.13 in \cite{CDH} for details.

Let now $X$ be a complete, finite rank median space. In Section~3 of \cite{Fioravanti1}, we endowed the set $\mscr{H}$ with a $\s$-algebra $\wh{\mscr{B}}$ and a measure $\wh{\nu}_X$ (usually denoted just $\wh{\nu}$). Unlike \cite{Fioravanti1}, here we simply refer to the elements of $\wh{\mscr{B}}$ as \emph{measurable sets}. 

Note that in general $\wh{\mscr{B}}$ and $\wh{\nu}$ differ from their counterparts in \cite{CDH}, in that $\wh{\mscr{B}}$ contains more measurable and null sets. More precisely, a subset $E\cu\mscr{H}$ is measurable/null if and only if every intersection with a halfspace interval is measurable/null. This results in the following useful properties.

The map $*\colon\mscr{H}\ra\mscr{H}$ sending each halfspace to its complement is measure preserving. Every inseparable subset of $\mscr{H}$ is measurable (Lemma~3.9 in \cite{Fioravanti1}); in particular, all ultrafilters are measurable and, for all ${x,y\in X}$, we have $\wh{\nu}(\s_x\triangle\s_y)=d(x,y)$. Almost every halfspace $\mf{h}\in\mscr{H}$ is \emph{thick}, i.e.\,both $\mf{h}$ and $\mf{h}^*$ have nonempty interior (Corollary~3.7 in \cite{Fioravanti1}). The next result, moreover, is Corollary~3.11 in \cite{Fioravanti1}.

\begin{prop}\label{M(X)=X}
Let $X$ be a complete, finite rank median space and $\s\cu\mscr{H}$ an ultrafilter such that $\wh{\nu}(\s\triangle\s_x)<+\infty$ for some $x\in X$. There exists $y\in X$ such that $\wh{\nu}(\s\triangle\s_y)=0$.
\end{prop}

As a consequence, $X$ can be equivalently described as the collection of all ultrafilters on $\mscr{H}$ that satisfy $\wh{\nu}(\s\triangle\s_x)<+\infty$ for some $x\in X$; we identify ultrafilters whose symmetric difference is $\wh{\nu}$-null. Considering instead the space of \emph{all} ultrafilters on $\mscr{H}$ we obtain a set $\overline X$ in which $X$ embeds. A structure of median algebra can be defined on $\overline X$ by setting
\[m(\s_1,\s_2,\s_3):=(\s_1\cap\s_2)\cup(\s_2\cap\s_3)\cup(\s_3\cap\s_1).\]
We endow $\overline X$ with a topology such that ultrafilters $\s_n\cu\mscr{H}$ converge to ${\s\cu\mscr{H}}$ if and only if  $\limsup(\s_n\triangle\s)$ is $\wh{\nu}$-null; see Section~4.2 and, in particular, Lemma~4.16 in \cite{Fioravanti1}. We refer to $\overline X$ as the \emph{Roller compactification}. 

Note that the identity map of $X$ canonically extends to a homeomorphism between $\overline X$ and the Busemann compactification of $X$ (Proposition~4.21 in \cite{Fioravanti1}). The difference in terminology is motivated by the additional median structure. The following is part of Proposition~4.14 in \cite{Fioravanti1}.

\begin{prop}
The Roller compactification $\overline X$ is a compact topological median algebra. The inclusion $X\hookrightarrow\overline X$ is a continuous morphism of median algebras with dense, convex image.
\end{prop}

If $X$ is connected and locally compact, $X$ is open in $\overline X$ and the inclusion $X\hookrightarrow\overline X$ is a homeomorphism onto its image (Proposition~4.20 in \cite{Fioravanti1}). This however fails in general. The \emph{Roller boundary} is defined as $\partial X:=\overline X\setminus X$. 

A point of $\overline X$ can often be represented by several distinct ultrafilters with null symmetric differences. However, for each $\xi\in\overline X$ there is a unique \emph{preferred} ultrafilter $\s_{\xi}$ representing $\xi$ (Lemma~4.15 in \cite{Fioravanti1}). This should be seen as a generalisation of the ultrafilters $\s_x$ when $x\in X$. 

We can extend each halfspace $\mf{h}$ of $X$ to a halfspace $\wt{\mf{h}}$ of $\overline X$ such that $\wt{\mf{h}}\cap X=\mf{h}$; indeed, it suffices to define $\wt{\mf{h}}:=\{\xi\in\overline X\mid\mf{h}\in\s_{\xi}\}$. When $\xi,\eta\in\overline X$, we save the notation $\mscr{H}(\xi|\eta)$ for the set $\s_{\eta}\setminus\s_{\xi}\cu\mscr{H}(X)$, instead of the analogous subset of $\mscr{H}(\overline X)$.

If $Y\cu X$ is a closed median subalgebra, the restriction of the metric of $X$ turns $Y$ into a complete median space with $\text{rank}(Y)\leq\text{rank}(X)$; moreover:

\begin{lem}\label{Roller of subalgebras}
There is a canonical morphism of median algebras $\iota_Y\colon\overline Y\hookrightarrow\overline X$. 
\end{lem} 
\begin{proof}
We write $\mscr{H}_Y:=\{\mf{h}\in\mscr{H}(X)\mid\mf{h}\cap Y\neq\emptyset,~\mf{h}^*\cap Y\neq\emptyset\}$; intersecting with $Y$ gives a map $p\colon\mscr{H}_Y\ra\mscr{H}(Y)$. Lemma~6.5 in \cite{Bow1} implies that $p$ is surjective. Thus, for every ultrafilter $\s\cu\mscr{H}(Y)$, there is a unique ultrafilter $\s'\cu\mscr{H}(X)$ such that $\s_Y\cu\s'$ and $p(\s'\cap\mscr{H}_Y)=\s$. Applying this to preferred ultrafilters yields the required embedding.
\end{proof}

Given ultrafilters $\s_1,\s_2\cu\mscr{H}$, we set $d(\s_1,\s_2):=\wh{\nu}(\s_1\triangle\s_2)$. We refer to $d$ as the \emph{extended metric} on $\overline X$ as it satisfies all the axioms of a metric, even though the value $+\infty$ is allowed. Note that, for points of $X$, this is the same as the original median metric on $X$. 

A \emph{component} $Z\cu\overline X$ is a maximal set of points having pairwise finite distances. Components are convex subsets of $\overline X$ (Proposition~4.19 in \cite{Fioravanti1}). One component always coincides with $X\cu\overline X$; all other components are contained in $\partial X$. For the following, see Propositions~4.23 and~4.29 in \cite{Fioravanti1}.

\begin{prop}\label{pi_Z 2}
Let $X$ be a complete median space of finite rank $r$. Let $Z\cu\partial X$ be a component and let $d$ denote the extended metric on $\overline X$.
\begin{enumerate}
\item The metric space $(Z,d)$ is a complete median space of rank at most $r-1$. 
\item Every thick halfspace of $Z$ is of the form $\wt{\mf{h}}\cap Z$ for a unique $\mf{h}\in\mscr{H}$.
\end{enumerate}
\end{prop}

If $C\cu X$ is closed and convex, the closure of $C$ in $\overline X$ is gate-convex and naturally identified with the Roller compactification of $C$ (see Proposition~4.14 in \cite{Fioravanti1}); thus the notation $\overline C$ is not ambiguous. We denote by $\pi_C\colon\overline X\ra\overline C$ the corresponding gate-projection; it extends the usual gate-projection $X\ra C$. If $\s\cu\mscr{H}(X)$ is an ultrafilter representing the point $\xi\in\overline X$, the set $\s\cap\mscr{H}(C)$ is an ultrafilter on $\mscr{H}(C)$ and represents $\pi_C(\xi)$. 

Similarly, if $Z\cu\partial X$ is a component, the closure of $Z$ in $\overline X$ is gate-convex and naturally identified with the Roller compactification $\overline Z$ (Proposition~4.28 in \cite{Fioravanti1}). The gate-projection $\pi_Z\colon\overline X\ra\overline Z$ satisfies $\pi_Z(X)\cu Z$. In terms of ultrafilters, $\pi_Z$ takes the point of $\overline X$ represented by $\s\cu\mscr{H}(X)$ to the point of $\overline Z$ represented by $\s\cap\mscr{H}(Z)\cu\mscr{H}(Z)$. The intersection makes sense as, by part~2 of Proposition~\ref{pi_Z 2}, almost every halfspace of $Z$ arises from a halfspace of $X$.

Let now $\G$ be a group and $\G\acts X$ an action by isometries. 

\begin{defn}
We say that $\G\acts X$ is \emph{Roller elementary} if there exists a finite orbit within $\overline X$. The action is \emph{Roller minimal} if $\text{rank}(X)\geq 1$ and $\G$ does not preserve any proper, closed, convex subset $C\cu\overline X$. 
\end{defn}

Roller elementarity implies -- but is in general much stronger than -- the existence of a finite orbit in the visual compactification of $\wh{X}$. When $X$ is a {\rm CAT}(0) cube complex, an action is Roller minimal if and only if, in the terminology of \cite{CS}, it is essential and does not fix any point in the visual boundary of $\wh{X}$. For these statements, see Proposition~5.2 in \cite{Fioravanti2}.

Neither Roller elementarity, nor Roller minimality implies the other one. However, Roller minimal actions naturally arise from Roller nonelementary ones (Proposition~5.5 in \cite{Fioravanti2}): 

\begin{prop}\label{Roller elementary vs strongly so}
Let $X$ be a complete, finite rank median space with an isometric action $\G\acts X$. Either $\G\acts\overline X$ fixes a point or there exist a $\G$-invariant component $Z\cu\overline X$ and a $\G$-invariant, closed, convex subset ${C\cu Z}$ such that $\G\acts C$ is Roller minimal.
\end{prop}

If $Z\cu\overline X$ is a component, any $\G$-invariant, closed, convex subset $C\cu Z$ gives rise to a measurable decomposition $\mscr{H}=\mscr{H}_C\sqcup(\s_C\cup\s_C^*)$. Here, we have introduced the sets ${\mscr{H}_C:=\{\mf{h}\in\mscr{H}\mid C\cap\wt{\mf{h}}\neq\emptyset,~C\cap\wt{\mf{h}}^*\neq\emptyset\}}$ and ${\s_C:=\{\mf{h}\in\mscr{H}\mid C\cu\wt{\mf{h}}\}}$. Note that, by part~2 of Proposition~\ref{pi_Z 2}, the measure spaces $(\mscr{H}_C,\wh{\nu}_X)$ and $(\mscr{H}(C),\wh{\nu}_C)$ are isomorphic.

We say that an action $\G\acts X$ is \emph{without wall inversions} if there do not exist $g\in\G$ and $\mf{h}\in\mscr{H}$ such that $g\mf{h}=\mf{h}^*$. By Proposition~\ref{trivial chains of halfspaces}, any action on a \emph{connected}, complete, finite rank median space is without wall inversions. The following appears as Corollary~5.4 in \cite{Fioravanti2}; compare with \cite{CS}.

\begin{prop}\label{double skewering}
Let $X$ be a complete, finite rank median space with thick halfspaces $\mf{h}\cu\mf{k}$. Let $\G\acts X$ be a Roller minimal action without wall inversions. 
\begin{enumerate}
\item There exists $g\in\G$ such that $g\mf{h}^*\subsetneq\mf{h}$ and ${d\left(g\mf{h}^*,\mf{h}^*\right)>0}$.
\item There exists $g\in\G$ such that $g\mf{k}\subsetneq\mf{h}\cu\mf{k}$ and $d(g\mf{k},\mf{h}^*)>0$.
\end{enumerate}
\end{prop}

When $G$ is a topological group, all isometric actions $G\acts X$ will be implicitly required to have continuous orbit maps. Equivalently, the homomorphism $G\ra\text{Isom}~X$ is continuous, where we endow $\text{Isom}~X$ with the topology of pointwise convergence. We remark that $\text{Isom}~X$ is a Hausdorff, sequentially complete topological group as soon as $X$ is complete.

If $(M_1,m_1)$ and $(M_2,m_2)$ are median algebras, the \emph{product median algebra} is defined as $(M_1\x M_2,m)$, where $m=(m_1\o p_1,m_2\o p_2)$; here $p_i$ denotes the projection onto the $i$-th factor. If $(X_1,d_1)$ and $(X_2,d_2)$ are median spaces, we endow the product $X_1\x X_2$ with the $\ell^1$ metric, namely $d=d_1\o p_1+d_2\o p_2$. The median algebra associated to the median space $(X_1\x X_2, d)$ is just the product median algebra arising from $X_1$ and $X_2$. 

A median space $X$ is said to be \emph{irreducible} if it is not isometric to any nontrivial product $X_1\x X_2$ of median spaces. The structure of reducible median spaces is described by the following result. 

\begin{prop}\label{products}
Let $X_1,...,X_k$ be irreducible, complete, finite rank median spaces; consider the product $X=X_1\x...\x X_k$. 
\begin{enumerate}
\item We have a measurable partition $\mscr{H}(X)=\mscr{H}_1\sqcup...\sqcup\mscr{H}_k$, where each $\mscr{H}_i$ is canonically identified with $\mscr{H}(X_i)$. If $\mf{h}\in\mscr{H}_i$ and $\mf{k}\in\mscr{H}_j$ with $i\neq j$, the halfspaces $\mf{h}$ and $\mf{k}$ are transverse. 
\item Every isometry of $X$ permutes the members of the partition. The product $\text{Isom}~X_1\x...\x\text{Isom}~X_k$ sits inside $\text{Isom}~X$ as the open, finite index subgroup preserving the splitting $X=X_1\x...\x X_k$.
\item Every closed, convex subset $C\cu X$ is of the form $C_1\x...\x C_k$, where each $C_i$ is a closed convex subset of $X_i$.
\item The Roller compactification $\overline X$ is naturally identified with the product median algebra $\overline{X_1}\x...\x\overline{X_k}$, endowed with the product topology.
\end{enumerate}
\end{prop}
\begin{proof}
Let $p_i\colon X\ra X_i$ denote the projection onto the $i$-th factor. Given $\mf{h}\in\mscr{H}(X_i)$, we have $p_i^{-1}(\mf{h})\in\mscr{H}(X)$; let $\mscr{H}_i$ be the set of halfspaces of $X$ that arise this way. With this notation, part~1 is Proposition~2.10 in \cite{Fioravanti2}. 

Regarding part~2, Proposition~2.12 in \cite{Fioravanti2} shows that each $g\in\text{Isom}~X$ permutes the sets $\mscr{H}_i$, or equivalently the factors $X_i$. Moreover, the subgroup of isometries preserving each $\mscr{H}_i$ (call it $H$) has finite index in $\text{Isom}~X$ and is naturally identified with the product $\text{Isom}~X_1\x...\x\text{Isom}~X_k$. We are left to prove that $H$ is open in $\text{Isom}~X$.

Choose distinct points $x_i,y_i\in X_i$ and a real number $\eps>0$ such that $2\eps< d(x_i,y_i)$ for all $i$. Let $x\in X$ be the point with coordinates $x_i$; we also introduce the points $z_i\in X$ such that $p_j(z_i)=x_j$ for all $i\neq j$ and $p_i(z_i)=y_i$. We will show that an isometry $F\in\text{Isom}~X$ lies in $H$ as soon as $d(F(x),x)<\eps$ and $d(F(z_i),z_i)<\eps$ for all $i$. This will yield that $H$ is open.

Suppose for the sake of contradiction that such an isometry $F$ does not lie in $H$; in particular, there exist indices $i\neq j$ such that $F(\mscr{H}_i)=\mscr{H}_j$. Since $X_i$ can be realised via ultrafilters on $\mscr{H}_i=\mscr{H}(X_i)$, this naturally induces an isometry $f\colon X_i\ra X_j$; cf.\,Corollary~3.13 in \cite{Fioravanti1}. Observe that $p_j\o F=f\o p_i$. Thus, we have $f(x_i)=fp_i(x)=p_jF(x)$ and
\[d(x_j,f(x_i))=d(p_j(x),p_jF(x))\leq d(x,F(x))<\eps.\]
Similarly, $d(x_j,f(y_i))=d(p_j(z_i),p_jF(z_i))\leq d(z_i,F(z_i))<\eps$. We conclude that $d(x_i,y_i)=d(f(x_i),f(y_i))<2\eps$, a contradiction.

We now prove part~3. Irreducibility of the factors plays no role here, so it suffices to consider the case $k=2$. Let $C_1$ and $C_2$ be the projections of $C$ to $X_1$ and $X_2$. If $x\in C_1$ and $y\in C_2$, there exist $u\in X_1$ and $v\in X_2$ such that the points $\overline x:=(x,v)$ and $\overline y:=(u,y)$ lie in $C$. It is immediate to observe that the point $(x,y)$ lies in $I(\overline x,\overline y)\cu C$. Finally, part~4 is Lemma~2.11 in \cite{Fioravanti2}.
\end{proof}

Given $n\geq 2$, halfspaces $\mf{h}_1,...,\mf{h}_n$ form a \emph{facing $n$-tuple} if they are pairwise disjoint. We say that $\mf{h},\mf{k}\in\mscr{H}$ are \emph{strongly separated} if $\overline{\mf{h}}\cap\overline{\mf{k}}=\emptyset$ and no $\mf{j}\in\mscr{H}$ is transverse to both $\mf{h}$ and $\mf{k}$. See Theorem~5.9 and Lemma~6.3 in \cite{Fioravanti2} for the following.

\begin{prop}\label{strong separation}
Let $X$ be an irreducible, complete, finite rank median space; let $\mf{h}$ be a thick halfspace.
\begin{enumerate}
\item If $X$ admits a Roller minimal action without wall inversions, there exist thick halfspaces $\mf{h}'\cu\mf{h}\cu\mf{h}''$ such that $\mf{h}'$ and $\mf{h}''^*$ are strongly separated.
\item If $X$ admits a Roller nonelementary, Roller minimal action without walls inversions, $\mf{h}$ is part of a facing $n$-tuple of thick halfspaces for every $n\geq 2$.
\end{enumerate}
\end{prop}

Every complete, finite rank median space can be isometrically embedded into its \emph{barycentric subdivision} $X'$. This is a complete median space of the same rank; see Section~2.3 in \cite{Fioravanti2}. When $X$ is the $0$-skeleton of a {\rm CAT}(0) cube complex with the $\ell^1$ metric, the space $X'$ is given by the $0$-skeleton of the customary barycentric subdivision. 

We have a natural homomorphism $\text{Isom}~X\hookrightarrow\text{Isom}~X'$. Given any isometric action $\G\acts X$, the induced action $\G\acts X'$ is without wall inversions. We write $(\mscr{H}',\nu')$ instead of $(\mscr{H}(X'),\nu_{X'})$. There is an inclusion preserving map $p\colon\mscr{H}'\ra\mscr{H}$; it is surjective, $(\text{Isom}~X)$-equivariant and its fibres have cardinality at most two. We have $\#(p^{-1}(\mf{h}))=2$ if and only if $\{\mf{h}\}$ is an atom for $\wh{\nu}$; in this case, we refer to each element of $p^{-1}(\mf{h})$ as a \emph{hemiatom}. The following is Lemma~5.7 in \cite{Fioravanti2}.

\begin{lem}\label{RNE for X'}
Let $X$ be a complete, finite rank median space. An action $\G\acts X$ is Roller elementary if and only if the induced action $\G\acts X'$ is.
\end{lem}

The sets $\{-1,1\}$ and $\{-1,0,1\}$ inherit a median algebra structure from the median space $\R$; in particular, we can consider the product median algebras $\{-1,1\}^k$ and $\{-1,0,1\}^k$ for every $k\geq 0$. For every point $x\in X'\setminus X$, there exists a canonical, gate-convex subset $\wh{C}(x)\cu X'$; it is isomorphic to $\{-1,0,1\}^k$, for some $k\geq 1$, via an isomorphism that takes $x$ to the centre $(0,...,0)$. The intersection $C(x):=\wh{C}(x)\cap X$ is gate-convex in $X$ and corresponds to the subset $\{-1,1\}^k\cu\{-1,0,1\}^k$. For $x\in X$, we set $\wh{C}(x)=C(x)=\{x\}$. See Lemma~2.13 in \cite{Fioravanti2} for more details.

\begin{lem}\label{unbounded convex in X'}
Let $X$ be a complete, finite rank median space. Every infinite, convex subset $E\cu X'$ intersects $X$.
\end{lem}
\begin{proof}
Let $x\in E$ be a point minimising $R:=\text{rank}(C(x))$; if $R=0$ we have $x\in E\cap X$. Otherwise, there exists a point $y\in E$ that does not lie in the finite set $\wh{C}(x)$; the gate-projection $z$ of $y$ to $\wh{C}(x)$ lies in $\wh{C}(x)\setminus\{x\}$. In particular $C(z)$ is contained in a face of $C(x)$ and has strictly lower rank, a contradiction since $z\in I(x,y)\cu E$.
\end{proof}

In particular, we obtain the following extension of Lemma~\ref{RNE for X'}.

\begin{lem}\label{RM for X'}
If $\G\acts X$ is Roller nonelementary and Roller minimal, so is the action $\G\acts X'$.
\end{lem}
\begin{proof}
Suppose for the sake of contradiction that $C\cu\overline{X'}$ is a nonempty, $\G$-invariant, closed, convex subset. By Corollary~4.31 in \cite{Fioravanti1} there exists a $\G$-invariant component $W\cu\overline{X'}$ with $W\cap C\neq\emptyset$. Note that $C\cap W$ is unbounded by Corollary~2.17 in \cite{Fioravanti2}, since $\G\acts X'$ is Roller nonelementary by Lemma~\ref{RNE for X'}. The component $W$ is the barycentric subdivision of a component $Z\cu\overline X$ and Lemma~\ref{unbounded convex in X'} implies that $C\cap Z\neq\emptyset$. Since $\G\acts X$ is Roller minimal, we must have $Z=X$ and $C\cap X=X$; hence $X'\cu C$ by part~2 of Proposition~2.15 in \cite{Fioravanti2}. We conclude that $C=\overline{X'}$. 
\end{proof}

Let us now fix a basepoint $x\in X$, where $X$ is a complete finite rank median space; the following discussion is independent of our choice of $x$. A \emph{diverging chain of halfspaces} is a sequence $(\mf{h}_n)_{n\geq 0}$ such that ${d(x,\mf{h}_n)\ra +\infty}$ and $\mf{h}_{n+1}\cu\mf{h}_n$ for each $n\geq 0$; we use the same terminology for the set $\{\mf{h}_n\mid n\geq 0\}$. The following notion will be key to our proof of Theorem~\ref{A} and to the discussion in Appendix~\ref{structure of UBS's}. It was first introduced by \cite{Hagen} in the context of {\rm CAT}(0) cube complexes. 

\begin{defn}\label{UBS defn}
Given $\xi\in\partial X$, a \emph{unidirectional boundary set (UBS)} for $\xi$ is an inseparable subset $\Om\cu\s_{\xi}\setminus\s_x$ containing a diverging chain of halfspaces. 
\end{defn}

Given UBS's $\Om_1,\Om_2\cu\s_{\xi}\setminus\s_x$, we say that $\Om_1$ is \emph{almost contained} in $\Om_2$ if the halfspaces in $\Om_1\setminus\Om_2$ lie at uniformly bounded distance from $x$; this is denoted by $\Om_1\preceq\Om_2$. If $\Om_1\preceq\Om_2$ and $\Om_2\preceq\Om_1$, the UBS's are \emph{equivalent} and we write $\Om_1\sim\Om_2$. We denote the equivalence class of $\Om\cu\mscr{H}$ by $[\Om]$ and the set of all equivalence classes of UBS's for $\xi$ by $\overline{\mc{U}}(\xi)$; the relation $\preceq$ descends to a partial order on $\overline{\mc{U}}(\xi)$. A UBS $\Om$ is said to be \emph{minimal} if $[\Om]$ is a minimal element of $(\overline{\mc{U}}(\xi),\preceq)$. A minimal UBS is equivalent to the inseparable closure of any diverging chain that it contains. 

We define a directed graph $\mc{G}(\xi)$ as follows; also cf.\,\cite{corrigendum,Fioravanti2}. The vertex set of $\mc{G}(\xi)$ is identified with the set of minimal elements of $(\overline{\mc{U}}(\xi),\preceq)$. Given diverging chains $(\mf{h}_m)_{m\geq 0}$ and $(\mf{k}_n)_{n\geq 0}$ in $\Om_1$ and $\Om_2$ respectively, we draw an oriented edge from $[\Om_1]$ to $[\Om_2]$ if almost every $\mf{h}_m$ is transverse to almost every $\mf{k}_n$, but not vice versa; this is independent of the choices involved. A subset $A\cu\mc{G}(\xi)^{(0)}$ is said to be \emph{inseparable} if every directed path between vertices in $A$ only crosses vertices in $A$. For the following, see Lemma~4.5 and Proposition~4.7 in \cite{Fioravanti2}.

\begin{prop}\label{main prop on UBS's}
Let $X$ be a complete median space of finite rank $r$ and $\xi\in\partial X$.
\begin{enumerate}
\item The graph $\mc{G}(\xi)$ has at most $r$ vertices and contains no directed cycles.
\item The poset $(\overline{\mc{U}}(\xi),\preceq)$ is isomorphic to the poset of inseparable subsets of $\mc{G}(\xi)^{(0)}$, ordered by inclusion. The isomorphism maps $[\Om]\in\overline{\mc{U}}(\xi)$ to the set of equivalence classes of minimal UBS's almost contained in $\Om$. In particular, the set $\overline{\mc{U}}(\xi)$ is finite.
\item Given a UBS $\Om$ and a set $\{\Om_1,...,\Om_k\}$ of representatives of all equivalence classes of minimal UBS's almost contained in $\Om$, we have
\[\sup_{\mf{h}\in\Om\triangle\left(\Om_1\cup ... \cup\Om_k\right)} d(x,\mf{h})<+\infty.\]
\end{enumerate}
\end{prop}

If $\Om$ is a minimal UBS, we denote by $\Om'$ the subset of halfspaces $\mf{k}\in\Om$ that are not transverse to any diverging chain of halfspaces in $\Om$. We say that $\Om$ is \emph{reduced} if $\Om=\Om'$. We say that $\Om$ is \emph{strongly reduced} if we can write $\Om=\mc{C}_1\sqcup ... \sqcup\mc{C}_k$ for some $k\geq 1$, where each $\mc{C}_i$ is totally ordered by inclusion and contains a diverging chain of halfspaces.

\begin{figure}
\centering
\includegraphics[width=2.75in]{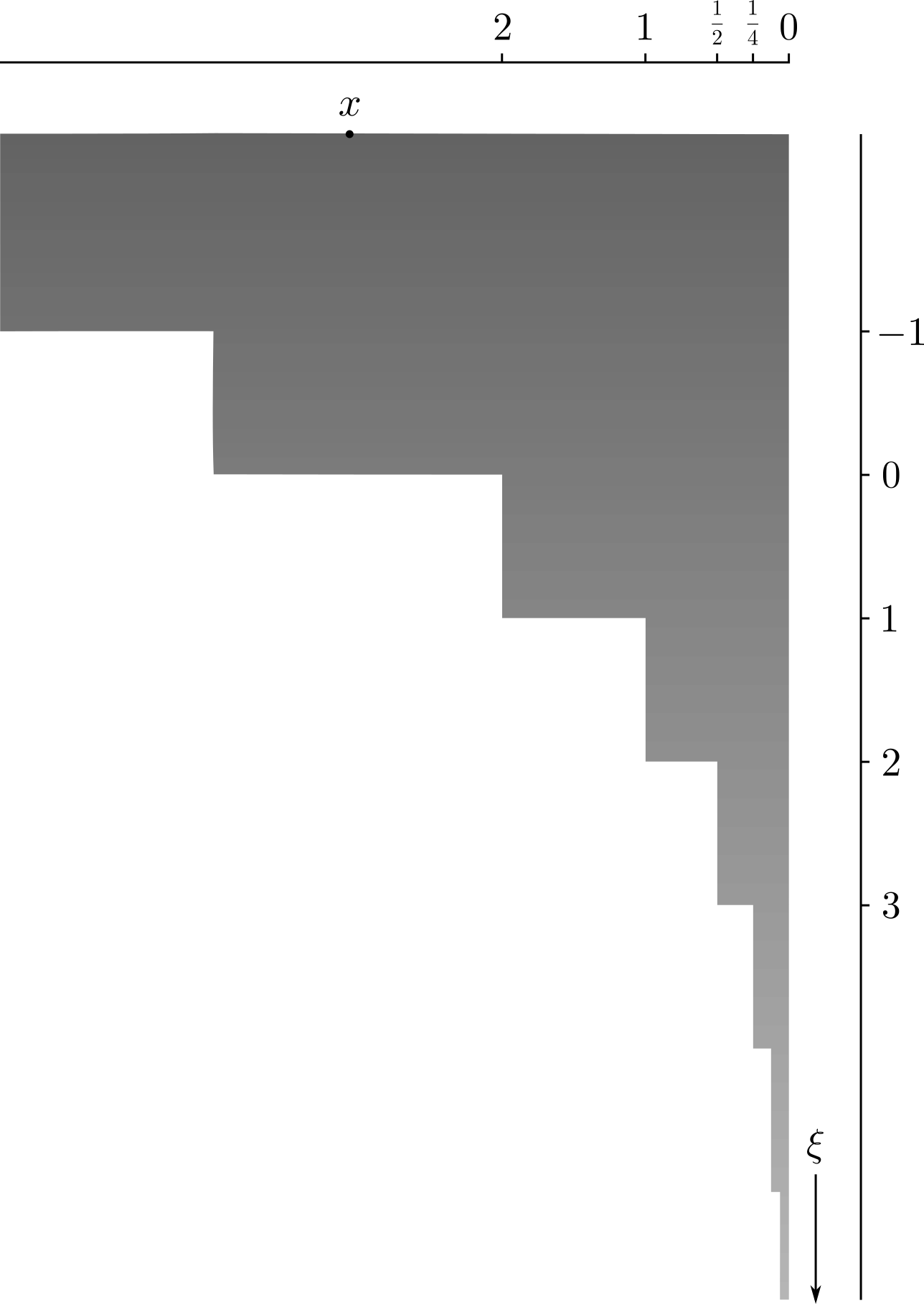}
\caption{}
\label{not strongly reduced}
\centering
\includegraphics[width=2.75in]{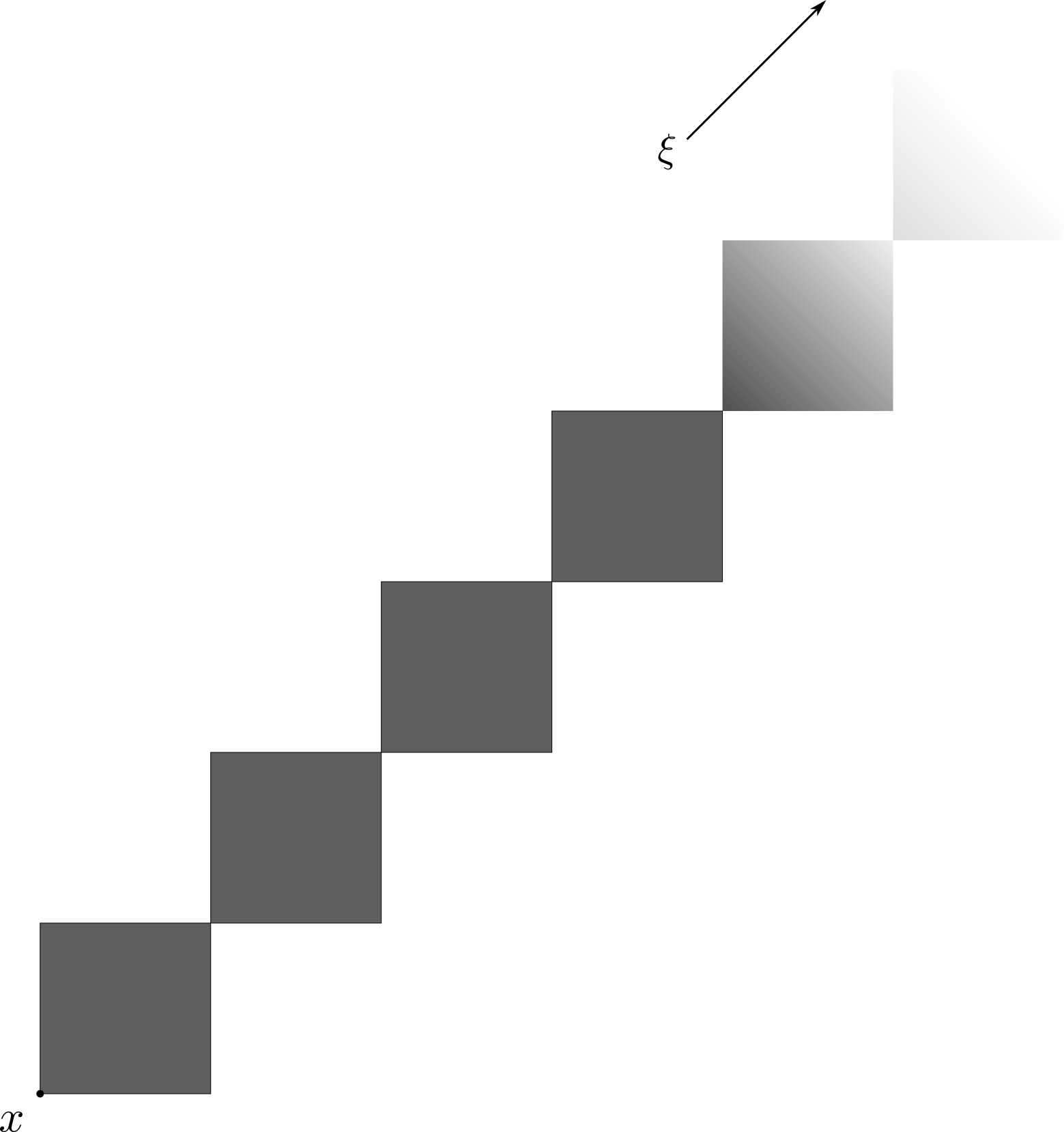}
\caption{}
\label{sr k=2}
\end{figure}

Consider the median spaces in Figures~\ref{not strongly reduced} and~\ref{sr k=2}; both are subsets of $\R^2$ with the restriction of the $\ell^1$ metric. In both cases, the UBS $\Om=\s_{\xi}\setminus\s_x$ is minimal. Figure~\ref{not strongly reduced} shows that $\Om$ can be reduced, but not strongly reduced. In Figure~\ref{sr k=2} the UBS $\Om$ is strongly reduced and exhibits how the decomposition $\Om=\mc{C}_1\sqcup ... \sqcup\mc{C}_k$ can require $k\geq 2$.

\begin{lem}\label{reduced and strongly reduced}
Let $X$ be a complete, finite rank median space. Consider $x\in X$ and $\xi\in\partial X$; let $\Om\cu\s_{\xi}\setminus\s_x$ be a minimal UBS.
\begin{enumerate}
\item The subset $\Om'$ is a reduced UBS equivalent to $\Om$. 
\item There exists a strongly reduced UBS contained in $\Om$; all its sub-UBS's are strongly reduced.
\item If $\Om$ is strongly reduced, it is reduced.
\end{enumerate}
\end{lem}
\begin{proof}
If $\mf{h}\cu\mf{k}$ lie in $\s_{\xi}\setminus\s_x$ and $\mf{k}$ is transverse to a diverging chain of halfspaces, then $\mf{h}$ is transverse to an infinite subchain. This implies that $\Om'$ is inseparable; moreover, $\Om\setminus\Om'$ cannot contain a diverging chain or $\Om$ would contain two inequivalent UBS's. This proves part~1.

To obtain part~2, we decompose $\Om=\mc{C}_1\sqcup ... \sqcup\mc{C}_k$ as in Lemma~\ref{Dilworth for differences}; let $A$ be the union of the sets $\mc{C}_i$ that do not contain a diverging chain. There exists $D<+\infty$ such that $d(x,\mf{h})\leq D$ for every $\mf{h}\in A$. The set $\{\mf{h}\in\Om\mid d(x,\mf{h})>D\}$ is a strongly reduced UBS and so are its sub-UBS's.

Regarding part~3, decompose $\Om=\mc{C}_1\sqcup ... \sqcup\mc{C}_k$, where each $\mc{C}_i$ is totally ordered by inclusion and contains a diverging chain. If there existed ${\mf{k}\in\Om\setminus\Om'}$, we would have $\mf{k}\in\mc{C}_i$ for some $i$; in particular, $\mf{k}$ would not be transverse to a diverging chain in $\mc{C}_i$. Since $\Om$ is minimal, $\mf{k}$ would not be transverse to any diverging chain in $\Om$, a contradiction.
\end{proof}

Given $\xi\in\partial X$, we denote by $\text{Isom}_{\xi}X$ the group of isometries that fix $\xi$. Let $K_{\xi}$ be the kernel of the action of $\text{Isom}_{\xi}X$ on $\overline{\mc{U}}(\xi)$; it is a finite-index subgroup of $\text{Isom}_{\xi}X$.

If $\Om\cu\s_{\xi}\setminus\s_x$ is a UBS, we denote by $K_{\Om}$ the subgroup of $\text{Isom}_{\xi}X$ that fixes the equivalence class $[\Om]$. We can define a \emph{transfer character} $\chi_{\Om}\colon K_{\Om}\ra\R$ by the formula $\chi_{\Om}(g)=\wh{\nu}\left(g^{-1}\Om\setminus\Om\right)- \wh{\nu}\left(\Om\setminus g^{-1}\Om\right)$; compare Section~4.H in \cite{Cor}. This is a homomorphism and only depends on the equivalence class $[\Om]$ (see Lemma~4.8 in \cite{Fioravanti2}). If $\Om_1,...,\Om_k$ are a set of representatives for $\mc{G}(\xi)^{(0)}$, we can also consider the \emph{full transfer homomorphism} $\chi_{\xi}=(\chi_{\Om_1},...,\chi_{\Om_k})\colon K_{\xi}\ra\R^k$. 

\begin{prop}\label{chi_xi continuous}
Let $X$ be a complete, finite rank median space; consider ${\xi\in\partial X}$. The subgroup $K_{\xi}$ is open in $\text{Isom}_{\xi}X$ and the full transfer homomorphism $\chi_{\xi}\colon K_{\xi}\ra\R^k$ is continuous. Every finitely generated subgroup of $\ker\chi_{\xi}$ has a finite orbit in $X$; if $X$ is connected, every finitely generated subgroup of $\ker\chi_{\xi}$ fixes a point.
\end{prop}

Before proving the proposition, we will need to obtain the following lemma. Note that, for every point $\xi\in\partial X$ and every halfspace $\mf{h}\in\s_{\xi}$, the set $\mscr{H}(\mf{h}^*|\xi)=\{\mf{k}\in\s_{\xi}\mid\mf{k}\cu\mf{h}\}$ is a UBS.

\begin{lem}\label{aux lemma 1}
For every thick halfspace $\mf{h}\in\s_{\xi}$ and every $\eps>0$, there exists a neighbourhood $U$ of the identity in $\text{Isom}_{\xi}X$ such that $\mscr{H}(\mf{h}^*|\xi)\sim\mscr{H}(g\mf{h}^*|\xi)$ and $\wh{\nu}\left(\mscr{H}(\mf{h}^*|\xi)\triangle\mscr{H}(g\mf{h}^*|\xi)\right)<\eps$ for all $g\in U$.
\end{lem}
\begin{proof}
Pick a point $x\in X$ with $d(x,\mf{h})>0$; in a neighbourhood of the identity of $\text{Isom}_{\xi}X$, we have $x\in g\mf{h}^*$. If $\mf{k}\in\mscr{H}(\mf{h}^*|\xi)\setminus\mscr{H}(g\mf{h}^*|\xi)$ and $y$ is the gate-projection of $x$ to $\overline{\mf{k}}$, we have $y\in g\overline{\mf{h}^*}$ by part~4 of Proposition~\ref{all about gates}, since $\overline{\mf{k}}\cap g\overline{\mf{h}^*}\neq\emptyset$ and $x\in g\overline{\mf{h}^*}$; thus $d(y,g^{-1}y)\geq d(\mf{k},\mf{h}^*)$. We conclude that, for every $\mf{k}\in\mscr{H}(\mf{h}^*|\xi)$ with $d(\mf{k},\mf{h}^*)>0$, there exists a neighbourhood $V_{\mf{k}}$ of the identity in $\text{Isom}_{\xi}X$ such that $\mf{k}\in\mscr{H}(g\mf{h}^*|\xi)$ for all $g\in V_{\mf{k}}$. \\
Decompose $\mscr{H}(\mf{h}^*|\xi)=\mc{C}_1\sqcup ... \sqcup\mc{C}_k$ as in Lemma~\ref{Dilworth for differences}. Let $\mf{k}_i$ be the union of all $\mf{k}\in\mc{C}_i$ with $d(\mf{h}^*,\mf{k})\geq\frac{\eps}{2k}$; if nonempty, it is a halfspace and $d(\mf{h}^*,\mf{k}_i)\geq\frac{\eps}{2k}$. Elements of $\mscr{H}(\mf{h}^*|\xi)$ not contained in any $\mf{k}_i$ form a subset of measure at most $\sum\wh{\nu}(\mscr{H}(\mf{h}^*|\mf{k}_i))\leq\frac{\eps}{2}$. Let $V$ be the intersection of the $V_{\mf{k}_i}$ for $\mf{k}_i\neq\emptyset$; if $g\in V$, the set $\mscr{H}(\mf{h}^*|\xi)\setminus\mscr{H}(g\mf{h}^*|\xi)$ has measure at most $\frac{\eps}{2}$ and consists of halfspaces at uniformly bounded distance from $x\in X$. It now suffices to consider $U:=V\cap V^{-1}$.
\end{proof}

\begin{proof}[Proof of Proposition~\ref{chi_xi continuous}.]
We only need to prove that $K_{\xi}$ is open and that $\chi_{\xi}$ is continuous; the rest of the statement is contained in Theorem~F of \cite{Fioravanti2}. For every $v\in\mc{G}(\xi)^{(0)}$, there exists $\mf{h}_v$ such that the vertices $w\in\mc{G}(\xi)^{(0)}$ with $w\preceq[\mscr{H}(\mf{h}_v^*|\xi)]$ are precisely $v$ and those that are at the other end of an incoming edge at $v$. Indeed, given a diverging chain in a UBS representing $v$, almost every halfspace in the chain can be chosen as $\mf{h}_v$. 

Let $A_i$ be the set of vertices $v\in\mc{G}(\xi)^{(0)}$ such that there exists no directed path of length $\geq i$ in $\mc{G}(\xi)$ that ends at $v$. Note that, by Proposition~\ref{main prop on UBS's}, we have ${\emptyset=A_0\subsetneq A_1\subsetneq ...\subsetneq A_k=\mc{G}(\xi)^{(0)}}$ for some $k\leq r$. We will show that the subgroup $K_i\leq\text{Isom}_{\xi}X$ that fixes $A_i\cu\mc{G}(\xi)^{(0)}$ pointwise is open in $\text{Isom}_{\xi}X$ and that, for every $[\Om]\in A_i$, the homomorphism ${\chi_{\Om}\colon K_i\ra\R}$ is continuous. We proceed by induction on $i$, setting $K_0:=\text{Isom}_{\xi}X$. 

The base step is trivial. If $i\geq 1$, let $A_i\setminus A_{i-1}=\{v_1,...,v_s\}$ and consider the halfspaces $\mf{h}_{v_1},...,\mf{h}_{v_s}$. Setting ${\Xi_j:=\mscr{H}(\mf{h}_{v_j}^*|\xi)}$, Lemma~\ref{aux lemma 1} provides a neighbourhood $U$ of $\id$ in $\text{Isom}_{\xi}X$ such that $g\Xi_j\sim\Xi_j$ for all $j$ and all $g\in U$. A minimal UBS almost contained in $\Xi_j$ projects to an element of $A_{i-1}$ or to $v_j$; hence, we have $gv_j=v_j$ for every $g\in U\cap K_{i-1}$. Since, by the inductive hypothesis, $K_{i-1}$ is open, so is $K_i$. Continuity of the transfer characters is obtained with a similar argument.
\end{proof}

\subsection{Bridges.}\label{bridges}

Let a median algebra $(M,m)$ and two gate-convex subsets $C_1$, $C_2$ be fixed throughout this section. All the following results have analogues in the context of {\rm CAT}(0) cube complexes; see Section~2.G in \cite{CFI}.

Denote by $\pi_i\colon M\ra C_i$ the gate-projection to $C_i$. We will refer to the sets
\[S_1:=\left\{x_1\in C_1\mid \exists x_2\in C_2 \text{ s.t. } (x_1,x_2) \text{ are gates for } (C_1,C_2)\right\}, \]
\[S_2:=\left\{x_2\in C_2\mid \exists x_1\in C_1 \text{ s.t. } (x_1,x_2) \text{ are gates for } (C_1,C_2)\right\}, \]
as the \emph{shores} of $C_1$ and $C_2$, respectively. By part~3 of Proposition~\ref{all about gates}, these coincide with $\pi_1(C_2)$ and $\pi_2(C_1)$, hence they are gate-convex. The map $\pi_2|_{S_1}\colon S_1\ra S_2$ is a bijection with inverse $\pi_1|_{S_2}$; if $M$ arises from a median space $X$, this is an isometry as gate-projections are $1$-Lipschitz. The \emph{bridge} is the set
\[
\tag{$\ast$}
B:=\bigsqcup_{x_1\in S_1} I\left(x_1,\pi_2(x_1)\right)=\bigsqcup_{x_2\in S_2} I\left(\pi_1(x_2),x_2\right).\]
The union is disjoint because, if $(x_1,x_2)$ is a pair of gates for $(C_1,C_2)$, we have $\pi_i(I(x_1,x_2))=\{x_i\}$ for $i=1,2$; this follows from part~4 of Proposition~\ref{all about gates} and the observation that $I(x_1,x_2)\cap C_i=\{x_i\}$.

\begin{prop}\label{all about bridges}
The bridge $B$ is gate-convex and 
\[\mscr{W}(B)=\left(\mscr{W}(C_1)\cap\mscr{W}(C_2)\right)\sqcup\mscr{W}(C_1|C_2).\]
For any pair of gates $(x_1,x_2)$, the bridge is canonically isomorphic to the product $S_1\times I(x_1,x_2)$; this is an isometry if $M$ arises from a median space.
\end{prop}
\begin{proof}
Pick a pair of gates $(x_1,x_2)$, set $I:=I(x_1,x_2)$ and consider the morphism of median algebras $\phi:=\pi_{S_1}\x\pi_I\colon M\ra S_1\x I$. If $(x_1',x_2')$ is another pair of gates, the projection $\pi_I$ provides an isomorphism ${I(x_1',x_2')\ra I}$ mapping each $x_i'$ to $x_i$. This observation and the decomposition $(\ast)$ above imply that the restriction $\phi|_B$ is bijective.  

The map $\pi_B=(\phi|_B)^{-1}\o\phi\colon M\ra B$ is surjective and it is a gate-projection by Proposition~2.1 in \cite{Fioravanti1}. By part~1 of Proposition~\ref{all about gates} and the discussion above, every wall of $B$ arises either from a wall of $M$ cutting $S_1$ or from a wall of $M$ cutting $I$; the latter correspond to $\mscr{W}(C_1|C_2)$ by part~2 of Proposition~\ref{all about gates}, so we are left to show that $\mscr{W}(S_1)=\mscr{W}(C_1)\cap\mscr{W}(C_2)$. This follows from the fact that $S_1=\pi_1(C_2)$.

When $M$ arises from a median space $X$, the measure $\wh{\nu}$ can be pushed forward via the standard two-to-one projection $\mscr{H}\ra\mscr{W}$. For the purpose of the current proof, we denote by $\wh{\mu}$ the resulting measure on $\mscr{W}$. The fact that $\phi|_B$ is an isometry then follows from the decomposition of $\mscr{W}(B)$ above and the observation that $\frac{1}{2}\cdot\wh{\mu}(\mscr{W}(x|y))=d(x,y)$ for all $x,y\in X$.
\end{proof}

We can extend the notion of strong separation to arbitrary gate-convex subsets of median algebras. We say that $C_1$ and $C_2$ are \emph{strongly separated} if they are disjoint and $\mscr{W}(C_1)\cap\mscr{W}(C_2)=\emptyset$. Note that the condition ${\mscr{W}(C_1)\cap\mscr{W}(C_2)=\emptyset}$ alone already implies that $C_1\cap C_2$ consists of at most one point. In a median space, two halfspaces are strongly separated in the sense of Section~\ref{prelims} if and only if their closures are strongly separated according to the definition above; see Lemma~\ref{ss vs closures} below for a stronger result. 

Proposition~\ref{all about bridges} implies that two disjoint, gate-convex sets are strongly separated if and only if their shores are singletons; this yields the following result.

\begin{cor}\label{strongly separated closed convex sets}
Let $C_1,C_2\cu M$ be strongly separated. There exists a unique pair of gates $(x_1,x_2)$ and $\pi_1(C_2)=\{x_1\}$, $\pi_2(C_1)=\{x_2\}$.
\end{cor}

We will also need the following:

\begin{lem}\label{ss vs closures}
Let $\mf{h},\mf{k}\in\mscr{H}$ be strongly separated as halfspaces. The closures $H,K$ of $\wt{\mf{h}},\wt{\mf{k}}$ in $\overline X$ are strongly separated as subsets of $\overline X$.
\end{lem}
\begin{proof}
Since $\mf{h},\mf{k}$ have disjoint closures in $X$, the sets $H$ and $K$ are disjoint by Helly's Theorem. By Proposition~\ref{all about bridges}, it then suffices to prove that the shore $S\cu H$ is a singleton. Suppose for the sake of contradiction that $S$ contains distinct points $\xi,\eta$ and let $\xi',\eta'\in K$ be their projections to $K$; in particular, $\s_{\eta}\setminus\s_{\xi}=\s_{\eta'}\setminus\s_{\xi'}$. Given $\mf{j}\in\s_{\eta}\setminus\s_{\xi}$, the argument at the beginning of the proof shows that the closures of $\mf{j}$ and $\mf{h}$ inside $X$ intersect nontrivially; by Lemma~3.6 in \cite{Fioravanti1}, almost every $\mf{j}\in\s_{\eta}\setminus\s_{\xi}$ intersects $\mf{h}$. Considering complements, we conclude that almost every $\mf{j}\in\s_{\eta}\setminus\s_{\xi}$ is transverse to $\mf{h}$ and, similarly, to $\mf{k}$. Since $d(\xi,\eta)>0$, there exists such a $\mf{j}$, contradicting the fact that $\mf{h},\mf{k}$ are strongly separated.
\end{proof}

\subsection{The Haagerup class.}\label{Haagerup}

Let $X$ be a median space, $G$ a topological group and $p\in [1,+\infty)$. Given a Banach space $E$, we denote by $\mscr{U}(E)$ the group of linear isometries of $E$. 

An isometric action $G\acts X$ corresponds to a measure preserving action $G\acts (\mscr{H},\wh{\nu})$. We obtain a continuous representation ${\rho_p\colon G\ra\mscr{U}(L^p(\mscr{H,\wh{\nu}}))}$; when $p=2$, we simply write $\rho:=\rho_2$. We will use interchangeably the notations $H^1_c(G,\rho_p)$ and $H^1_c\left(G,L^p(\mscr{H},\wh{\nu})\right)$ to denote continuous cohomology. 

Given $x\in X$, we can consider the continuous $1$-cocycle ${b^x\colon G\ra L^p(\mscr{H,\wh{\nu}})}$ defined by ${b^x(g):=g\cdot\mathds{1}_{\s_x}-\mathds{1}_{\s_x}}$; it satisfies ${\|b^x(g)\|_p=(2\cdot d(x,gx))^{1/p}}$. We will refer to $b^x$ as a \emph{Haagerup cocycle}. The cohomology class $[b^x]\in H_c^1(G,\rho_p)$ does not depend on the point $x$ and we will simply denote it by $[b]$. 

The action $G\acts X$ has bounded orbits if and only if the affine action $G\acts L^p(\mscr{H},\wh{\nu})$ induced by $b^x$ fixes a point; this follows for instance from the Ryll-Nardzewski Theorem for $p>1$ and from Theorem~A in \cite{Bader-Gelander-Monod} in the case $p=1$. Thus, we have $[b]=0$ if and only if the action $G\acts X$ has bounded orbits.

We can also consider the projection $\overline{[b]}$ of $[b]$ to \emph{reduced} continuous cohomology, which carries more interesting geometrical information (see Theorem~\ref{A} in the introduction). We will refer to $\overline{[b]}\in\overline{H_c^1}(G,\rho)$ as the \emph{Haagerup class} of $G\acts X$. The choice of $p=2$ here is not particularly relevant and the same discussion could be equally carried out for any other $p\in [1,+\infty)$ with few complications; see Remark~\ref{p neq 2}.

We conclude this section by collecting a few straightforward lemmata for later use. Let $\mc{H}$ be a Hilbert space and $G\ra\mscr{U}(\mc{H})$ a continuous unitary representation. 

\begin{lem}\label{H^1 vs open subgroups}
If $H\leq G$ is open and finite-index, functoriality induces an injective map $\overline{H^1_c}\left(G,\mc{H}\right)\hookrightarrow\overline{H^1_c}\left(H,\mc{H}\right)$.
\end{lem}

\begin{lem}\label{H^1 vs subrepresentations}
Given a $G$-invariant decomposition $\mc{H}=\mc{H}_1\perp\mc{H}_2$, the projections onto the two factors induce $\overline{H^1_c}\left(G,\mc{H}\right)\simeq\overline{H^1_c}\left(G,\mc{H}_1\right)\oplus\overline{H^1_c}\left(G,\mc{H}_2\right)$.
\end{lem}

Recall that we denote by $X'$ the barycentric subdivision of $X$ and by $i\colon X\hookrightarrow X'$ the standard inclusion; we write $(\mscr{H}',\nu')$ instead of $(\mscr{H}(X'),\wh{\nu}_{X'})$. Every isometric action $G\acts X$ also induces a continuous representation ${\rho'\colon G\ra\mscr{U}(L^2(\mscr{H}',\nu'))}$.

\begin{lem}\label{Haagerup vs X'}
Let $X$ be complete and finite rank and let $G\acts X$ be an isometric action. The projection $p\colon\mscr{H}'\ra\mscr{H}$ induces an isometric embedding $p^*\colon L^2\left(\mscr{H},\wh{\nu}\right)\hookrightarrow L^2\left(\mscr{H}',\nu'\right)$ and a monomorphism 
\[p_*\colon\overline{H^1_c}(G,\rho)\hookrightarrow\overline{H^1_c}(G,\rho')\] 
taking the Haagerup class of $G\acts X$ to the Haagerup class of $G\acts X'$.
\end{lem}
\begin{proof}
The fact that $p^*$ is an isometric embedding follows from the observation that $p_*\nu'=\wh{\nu}$. Injectivity of $p_*$ follows from Lemma~\ref{H^1 vs subrepresentations} applied to $L^2\left(\mscr{H},\wh{\nu}\right)$ and its orthogonal complement. Finally, if $x\in X$ and $b^x$ is the corresponding Haagerup cocycle for $G\acts X$, the cocycle $p^*\o b^x$ is the Haagerup cocycle for $G\acts X'$ relative to the point $i(x)\in X'$.
\end{proof}

\section{Haagerup class and elementarity of actions.}

\subsection{The main statement.}\label{Haag main}

Let $X$ be a complete, finite rank median space and let $G\acts X$ be an isometric action of a topological group $G$. The goal of this section is to prove Theorem~\ref{A}.

By Lemmata~\ref{RNE for X'} and~\ref{Haagerup vs X'}, it suffices to consider the case when $G\acts X$ is without wall inversions; this will be a standing assumption throughout the rest of the section.

\begin{lem}\label{no AIV's}
Suppose that $X$ is irreducible and that $G\acts X$ is Roller minimal and Roller nonelementary. There exists a non-abelian free subgroup ${H\leq G}$ such that $\rho$ has no $H$-almost-invariant vectors and $H\acts X$ has unbounded orbits.
\end{lem}
\begin{proof}
Part~2 of Proposition~6.4 in \cite{Fioravanti2} provides a non-abelian free subgroup $H\leq G$ and a measurable, $*$-invariant partition ${\mscr{H}=\bigsqcup_{h\in H}\mscr{H}_h}$ with $g\mscr{H}_h=\mscr{H}_{gh}$ for all $g,h\in H$. It is immediate from the construction of $H$ that it acts on $X$ with unbounded orbits. If there existed a sequence of almost invariant vectors $(F_n)_{n\geq 0}$ in $L^2(\mscr{H},\wh{\nu})$, say with $\|F_n\|_2=1$, we could define functions $f_n\in\ell^2(H)$ by $f_n(h):=\|F_n\mathds{1}_{\mscr{H}_h}\|_2$. It is immediate to check that $\|f_n\|_2=1$ for every $n\geq 0$ and, for every $g\in H$,
\begin{align*}
\|gf_n-f_n\|_2^2&=\sum_{h\in H}\left(\|F_n\mathds{1}_{\mscr{H}_{g^{-1}h}}\|_2-\left\|F_n\mathds{1}_{\mscr{H}_h}\right\|_2\right)^2 \\
&= \sum_{h\in H}\left(\left\|\left(gF_n\right)\mathds{1}_{\mscr{H}_{h}}\right\|_2-\left\|F_n\mathds{1}_{\mscr{H}_h}\right\|_2\right)^2 \\
&\leq\sum_{h\in H}\left\|\left(gF_n\right)\mathds{1}_{\mscr{H}_{h}}-F_n\mathds{1}_{\mscr{H}_h}\right\|_2^2=\left\|gF_n-F_n\right\|_2^2\xrightarrow[n\ra+\infty]{} 0.
\end{align*}
Thus, the regular representation of $H$ would contain almost-invariant vectors, implying amenability of $H$ (see e.g.\,Theorem~G.3.2 in \cite{Bekka-Harpe-Valette}). This is a contradiction.
\end{proof}

We can already prove the ``only if'' half of Theorem~\ref{A}.

\begin{prop}\label{Haagerup easy arrow}
If $G\acts X$ is Roller nonelementary, we have $\overline{[b]}\neq 0$.
\end{prop}
\begin{proof}
Note that, by functoriality of reduced cohomology, it suffices to consider the case when $G$ has the discrete topology; thus, we do not need to worry about continuity issues. We proceed by induction on $r=\text{rank}(X)$. When $r=0$, all actions are Roller elementary; suppose for the rest of the proof that $r\geq 1$.

We can assume that $G\acts X$ is Roller minimal. Indeed, if ${C\cu Z\cu\overline X}$ are the subsets provided by Proposition~\ref{Roller elementary vs strongly so}, the action $G\acts C$ is again without wall inversions and $\text{rank}(C)\leq r$, by Proposition~\ref{pi_Z 2}. The $G$-equivariant, measurable partition $\mscr{H}=\mscr{H}_C\sqcup(\s_C\cup\s_C^*)$ induces an orthogonal decomposition of $L^2\left(\mscr{H},\wh{\nu}\right)$ and a $G$-equivariant splitting
\[\overline{H^1_c}\left(G,L^2(\mscr{H},\wh{\nu})\right)=\overline{H^1_c}\left(G,L^2(\mscr{H}(C),\wh{\nu}_C)\right)\oplus\overline{H^1_c}\left(G,L^2(\s_C\cup\s_C^*,\wh{\nu})\right).\]
If $p_1$ and $p_2$ are the orthogonal projections of $L^2\left(\mscr{H},\wh{\nu}\right)$ onto the two direct summands, we can write $\overline{[b^x]}=\overline{[p_1b^x]}+\overline{[p_2 b^x]}$ for every $x\in X$. Note that the gate-projection $\pi_C\colon\overline X\ra\overline C$ maps $x$ to a point $\xi\in C$. The cocycle $p_1b^x$ is precisely the Haagerup cocycle $b^{\xi}$ for the action $G\acts C$, by part~2 of Proposition~\ref{pi_Z 2}. In particular, if $\overline{[b^{\xi}]}\neq 0$, we can conclude that $\overline{[b]}\neq 0$.

We thus assume in the rest of the proof that $X=C$, i.e.\,that $G\acts X$ is Roller minimal. If $X$ is irreducible, Lemma~\ref{no AIV's} provides $i\colon H\hookrightarrow G$ such that $H\acts X$ has unbounded orbits and $\rho$ has no $H$-almost-invariant vectors. The first condition implies that $i^*[b]\neq 0$; the second condition and Guichardet's trick (Theorem~1 in \cite{Guichardet}) thus yield $i^*\overline{[b]}\neq 0$. In particular, $\overline{[b]}\neq 0$.

If instead $X$ splits as a nontrivial product $X_1\x X_2$, there exists a finite-index subgroup $j\colon G_0\hookrightarrow G$ preserving this decomposition (see part~2 of Proposition~\ref{products}). It suffices to show that $j^*\overline{[b]}\neq 0$. Writing $\wh{\nu}_i$ instead of $\wh{\nu}_{X_i}$, Proposition~\ref{products} and Lemma~\ref{H^1 vs subrepresentations} imply:
\[\overline{H^1_c}\left(G_0,L^2(\mscr{H},\wh{\nu})\right)=\overline{H^1_c}\left(G_0,L^2(\mscr{H}(X_1),\wh{\nu}_1)\right)\oplus\overline{H^1_c}\left(G_0,L^2(\mscr{H}(X_2),\wh{\nu}_2)\right);\]
hence, if $x=(x_1,x_2)$, we have $\overline{[b^x]}=\overline{[b^{x_1}_{X_1}]}+\overline{[b^{x_2}_{X_2}]}$. The action $G_0\acts X$ is Roller nonelementary, since $G_0$ is finite-index in $G$; thus, up to exchanging the two factors, $G_0\acts X_1$ is Roller nonelementary. Since $\text{rank}(X_1)<r$, the inductive hypothesis guarantees that $\overline{[b^{x_1}_{X_1}]}\neq 0$ and this concludes the proof.
\end{proof} 

Before proving the rest of Theorem~\ref{A} we need to obtain a few more results.

\begin{lem}\label{open stabilisers}
If the $G$-orbit of $\xi\in\overline X$ is finite, the $G$-stabiliser of $\xi$ is open.
\end{lem}
\begin{proof}
Suppose that $G\xi=\{\xi\}\sqcup\{\xi_1,...,\xi_k\}$ and $d(\xi,\xi_i)\geq\eps>0$ for all $i$, where $d$ is the extended metric on $\overline X$. By Proposition~4.24 in \cite{Fioravanti1}, we can find $x_i,y_i\in X$ such that, denoting by $\pi_i$ the projection to $I_i:=I(x_i,y_i)$, we have $d(\pi_i(\xi),\pi_i(\xi_i))>\frac{3}{4}\eps$. In a neighbourhood $U\cu G$ of the identity element we have $d(gx_i,x_i)<\frac{1}{4}\eps$, $d(gy_i,y_i)<\frac{1}{4}\eps$ and $d(g\pi_i(\xi),\pi_i(\xi))<\frac{1}{4}\eps$, for all $i$.

If $g\in U$, we have $d(\pi_{gI_i}(\xi_i),\pi_i(\xi_i))\leq d(gx_i,x_i)+d(gy_i,y_i)<\frac{1}{2}\eps$; if in addition we had $g\xi=\xi_i$, we would have $\pi_{gI_i}(\xi_i)=g\pi_i(\xi)$. As a consequence, $d(\pi_i(\xi),\pi_i(\xi_i))\leq d(\pi_i(\xi),g\pi_i(\xi))+d(g\pi_i(\xi),\pi_i(\xi_i))<\frac{3}{4}\eps$, which would contradict our choice of $I_i$. We conclude that $U$ is contained in the stabiliser of $\xi$, which must be open in $G$.
\end{proof}

The proof of the following fact is rather lengthy and technical; it will be carried out in Appendix~\ref{structure of UBS's}. See Definition~\ref{UBS defn} for the definition of UBS.

\begin{prop}\label{point x_S}
Let $\xi\in\partial X$ and $K\cu\text{Isom}_{\xi}X$ be a compact set of isometries acting trivially on $\overline{\mc{U}}(\xi)$. There exists a point $x_K\in X$ such that $\s_{\xi}\setminus\s_{x_K}$ coincides, up to a null set, with $\Om_K:=\Om_K^1\sqcup ... \sqcup\Om_K^k$, where 
\begin{itemize}
\item each $\Om_K^i$ is a strongly reduced, minimal UBS; 
\item if $g\in K$ we have $g\Om_K^i\cu\Om_K^i$ whenever $\chi_{\Om_K^i}(g)\geq 0$ and $g\Om_K^i\supseteq\Om_K^i$ whenever $\chi_{\Om_K^i}(g)\leq 0$;
\item if $i\neq j$ and $g\in K$, we have $\Om_K^i\cap g\Om_K^j=\emptyset$. 
\end{itemize}
\end{prop}

Given points $\xi\in\partial X$, $x\in X$ and a UBS $\Om\cu\s_{\xi}\setminus\s_x$, we can define a function $\alpha_{\Om}\colon\Om\ra\R$ as $\alpha_{\Om}(\mf{h}):=\wh{\nu}\left(\mscr{H}(x|\mf{h})\cap\Om\right)$. The dependence on the point $x$ is not particularly relevant, so we do not record it in our notation. We can consider the sets ${\Om_c:=\{\mf{h}\in\Om\mid\alpha_{\Om}(\mf{h})\leq c\}}$. In Appendix~\ref{structure of UBS's} we will also obtain the following result (see Lemma~\ref{properties of Om_c 4}).

\begin{prop}\label{functions along UBS's}
Suppose that $\Om$ is minimal and strongly reduced. Let $K\cu\text{Isom}_{\xi}X$ be a compact set of isometries such that $g\Om\cu\Om$ for all $g\in K$. As $c\ra +\infty$, the functions
\[(g-\id)\cdot\left[-\left(1-\frac{\alpha_{\Om}}{c}\right)\cdot\mathds{1}_{\Om_c}\right]\]
converge to $\mathds{1}_{\Om\setminus g\Om}$ in $L^2\left(\mscr{H},\wh{\nu}\right)$, uniformly in $g\in K$. If instead $g\Om\supseteq\Om$ for all $g\in K$, they converge to the function $-\mathds{1}_{g\Om\setminus\Om}$.
\end{prop}

We are finally ready to complete the proof of Theorem~\ref{A}.

\begin{proof}[Proof of Theorem~\ref{A}]
By Proposition~\ref{Haagerup easy arrow}, it suffices to consider the case when $G$ has a finite orbit in $\overline X$ and, by Lemmata~\ref{H^1 vs open subgroups} and~\ref{open stabilisers}, we can actually assume that $G$ fixes a point $\xi\in\overline X$. If $\xi\in X$, we have $[b]=0$; suppose instead that $\xi\in\partial X$. By Proposition~\ref{chi_xi continuous}, an open, finite-index subgroup $G_0\leq G$ acts trivially on $\overline{\mc{U}}(\xi)$; by Lemma~\ref{H^1 vs open subgroups}, it suffices to consider the case $G=G_0$.

Fix $x\in X$. For every $\eps>0$ and every compact subset $K\cu G$, we need to construct a function $\psi\in L^2\left(\mscr{H}, \wh{\nu}\right)$ such that ${\left\|b^x(g)-(g\psi-\psi)\right\|_2<\eps}$ for all $g\in K$. Considering the point $x_K\in X$ provided by Proposition~\ref{point x_S}, it suffices to find a function $\phi\in L^2\left(\mscr{H}, \wh{\nu}\right)$ such that ${\left\|b^{x_K}(g)-(g\phi-\phi)\right\|_2<\eps}$ for all $g\in K$ and then set $\psi:=\phi+\mathds{1}_{\s_x}-\mathds{1}_{\s_{x_K}}$. If $g\in K$, considering all equalities up to null sets, we have:
\begin{align*}
\s_{gx_K}\setminus\s_{x_K}&=\left[\left(\s_{gx_K}\setminus\s_{x_K}\right)\cap\s_{\xi}\right]\sqcup\left[\left(\s_{gx_K}\setminus\s_{x_K}\right)\cap\s_{\xi}^*\right] \\
&=\left[\left(\s_{\xi}\setminus\s_{x_K}\right)\setminus\left(\s_{\xi}\setminus\s_{gx_K}\right)\right]\sqcup\left[\left(\s_{\xi}\setminus\s_{gx_K}\right)\setminus\left(\s_{\xi}\setminus\s_{x_K}\right)\right]^* \\
&=\left(\Om_K\setminus g\Om_K\right)\sqcup\left(g\Om_K\setminus\Om_K\right)^*.
\end{align*}
In particular, since by construction $\Om_K^i\cap g\Om_K^j=\emptyset$ whenever $i\neq j$,
\[\s_{gx_K}\setminus\s_{x_K}=\bigsqcup_{i=1}^k\left(\Om_K^i\setminus g\Om_K^i\right)\sqcup\left(g\Om_K^i\setminus\Om_K^i\right)^*.\]
Introducing the notation $\mathbbm{2}_A:=\mathds{1}_A-\mathds{1}_{A^*}$ for subsets $A\cu\mscr{H}$, we can rewrite:
\[b^{x_K}(g)=\mathbbm{2}_{\s_{gx_K}\setminus\s_{x_K}}=\sum_{\chi_{\Om_K^i}(g)>0}\mathbbm{2}_{\Om_K^i\setminus g\Om_K^i} -\sum_{\chi_{\Om_K^i}(g)<0}\mathbbm{2}_{g\Om_K^i\setminus\Om_K^i}. \]
For $c>0$, consider the function:
\[G_c:=\sum_{i=1}^k-\left(1-\frac{\alpha_{\Om^i_K}}{c}\right)\mathbbm{2}_{\left(\Om^i_K\right)_c}.\]
Proposition~\ref{functions along UBS's} shows that it suffices to take $\phi=G_c$ for large $c$. 
\end{proof}

\begin{rmk}\label{p neq 2}
Theorem~\ref{A} also holds for the analogous class in $\overline{H^1_c}(G,\rho_p)$, for every $p\in [1,+\infty)$. Indeed, Lem\-ma~\ref{H^1 vs subrepresentations} applies to any decomposition of a Banach space into closed subspaces. In the proof of Lemma~\ref{Haagerup vs X'}, a closed complement to $L^p(\mscr{H},\wh{\nu})$ within $L^p(\mscr{H}',\nu')$ is always provided by the subspace of functions on $\mscr{H}'$ that take opposite values on hemiatoms. Theorem~1 in \cite{Guichardet} holds for representations into general Banach spaces (although this does not appear in \cite{Guichardet}). Finally, if $F$ is a non-abelian free group, $\ell^p(F)$ has no $F$-almost-invariant vectors for every $p\in [1,+\infty)$.

The value of $p$ also has little importance for most of the material in Appendix~\ref{structure of UBS's}. Note however that Proposition~\ref{functions along UBS's} fails for $p=1$; one needs to consider functions with a quicker decay in that case.
\end{rmk}

\subsection{Elementarity and Shalom's property $H_{FD}$.}

Let $X$ be a complete, finite rank median space and let $G$ be a topological group acting on $X$ by isometries. Here we prove Theorem~\ref{Z} and Corollaries~\ref{G} and~\ref{H}.

We will need the following ergodicity result for the action $G\acts\mscr{H}$.

\begin{lem}\label{invariant sets}
Suppose that $X$ is irreducible and let $G\acts X$ be a Roller nonelementary, Roller minimal action. Let $E\cu\mscr{H}$ be a measurable subset such that $\wh{\nu}(gE\triangle E)=0$ for all $g\in G$. Then $\wh{\nu}(E)$ is either $0$ or $+\infty$.
\end{lem}
\begin{proof}
Without loss of generality we can assume that $G\acts X$ is without wall inversions, as Lemma~\ref{RM for X'} allows us to pass to the barycentric subdivision $X'$ if necessary. Now, suppose that $E$ is such a set and $0<\wh{\nu}(E)<+\infty$. Since $X$ has finite rank, we can find a thick halfspace $\mf{h}$ such that, replacing $E$ with $E^*$ if necessary, the set $E_{\mf{h}}:=E\cap\left\{\mf{k}\in\mscr{H}\mid \mf{k}\cu\mf{h}\right\}$ satisfies ${a:=\wh{\nu}(E_{\mf{h}})>0}$. By part~2 of Proposition~\ref{strong separation}, the halfspace $\mf{h}$ is part of a facing $n$-tuple with ${n>\frac{1}{a}\cdot\wh{\nu}(E)}$. By Proposition~\ref{double skewering}, there exist ${g_2,...,g_n\in G}$ such that $\mf{h},g_2\mf{h},...,g_n\mf{h}$ are a facing $n$-tuple. The sets $E_{\mf{h}}, g_2E_{\mf{h}},...,g_nE_{\mf{h}}$ are pairwise disjoint and contained in $E$, up to null sets. However, their union has measure $na>\wh{\nu}(E)$, a contradiction.
\end{proof}

\begin{proof}[Proof of Theorem~\ref{Z}]
Suppose for the sake of contradiction that $G\acts X$ is Roller nonelementary. Without loss of generality, we can assume that $X$ has minimal rank $r$ among complete median spaces admitting Roller nonelementary actions of $G$. In particular, $X$ must be irreducible, see the proof of Proposition~\ref{Haagerup easy arrow}. By Proposition~\ref{Roller elementary vs strongly so}, we can also assume that $G\acts X$ is Roller minimal. Theorem~\ref{A} guarantees that $\overline{H^1_c}(G,\rho)\neq\{0\}$ and, since $G$ has property $H_{FD}$, there exists a finite dimensional subrepresentation $V<L^2(\mscr{H},\wh{\nu})$. We will construct a measurable $G$-invariant subset of $\mscr{H}$ with positive, finite measure, thus violating Lemma~\ref{invariant sets}.

Let $f_1,...,f_k$ be measurable functions on $\mscr{H}$ whose equivalence classes in $L^2(\mscr{H},\wh{\nu})$ form an orthonormal basis of $V$. Define, for $c>0$,
\[E_c:=\left\{\mf{h}\in\mscr{H}\mid \exists \alpha=(\alpha_1,...,\alpha_k)\in\mbb{S}^{k-1} \text{ s.t. } \left|(\alpha_1f_1+...+\alpha_kf_k)(\mf{h})\right|> c \right\}.\]
Since in the definition of $E_c$ it suffices to look at $\alpha$'s lying in a countable dense subset of $\mbb{S}^{k-1}$, we conclude that $E_c$ is measurable. If $\mf{h}\in E_c$, we must have $|f_i(\mf{h})|>\frac{c}{k}$ for some $i$, hence $\wh{\nu}(E_c)<+\infty$; if $c$ is sufficiently small, we have $\wh{\nu}(E_c)>0$. Given $g\in G$, there exist real numbers $\alpha_{ij}$, $1\leq i,j\leq k$, such that, outside a measure zero set, we have $f_i=\sum_j\alpha_{ij}(gf_j)$ for every $i$. If $\mf{h}\in E_c\setminus gE_c$, we must have $f_i(\mf{h})\neq\sum_j\alpha_{ij}(gf_j)(\mf{h})$ for some $i$; we conclude that $\wh{\nu}(gE_c\triangle E_c)=0$ for all $g\in G$. 
\end{proof}

\begin{proof}[Proof of Corollary~\ref{G}]
Cocompactness of the action implies that $X$ is finite dimensional. By Proposition~\ref{chi_xi continuous} and Theorem~\ref{Z}, there exists a finite-index subgroup $\G_0\leq\G$ and a normal subgroup $N\lhd\G_0$ consisting of elliptic elements, such that $\G_0/N$ is abelian. Since $\G$ acts freely, $N$ is trivial.
\end{proof}

Recall that, in Gromov's density model, random groups at density $d<\frac{1}{2}$ are nonelementary hyperbolic with overwhelming probability \cite{asinv, Ollivier}. Together with Theorem~10.4 in \cite{Ollivier-Wise}, Corollary~\ref{G} then immediately implies Corollary~\ref{H}.

\section{Superrigidity.}\label{SR}

\subsection{The superrigidity result.}\label{SR section}

Let $X$ be a complete median space of finite rank $r$ and $\G\acts X$ an action by isometries of a discrete group $\G$.

\begin{lem}\label{strongly separated closed convex sets 2}
Suppose $\mf{h}_1,\mf{h}_2,\mf{h}_3\in\mscr{H}$ form a facing triple; let $\mf{k}_i\in\mscr{H}$ and $\mf{h}_i^*$ be strongly separated for $i=1,2,3$. There exists a point $z\in X$ such that $m(\xi_1,\xi_2,\xi_3)=z$ whenever $\xi_i\in\wt{\mf{k}_i}$. 
\end{lem}
\begin{proof}
Let $C$ be the intersection of the closures of $\wt{\mf{h}}_i^*$ inside $\overline X$; it is nonempty, closed and convex. Given points $\xi_i\in\wt{\mf{k}_i}$, set $m:=m(\xi_1,\xi_2,\xi_3)$. By convexity we have $I(\xi_2,\xi_3)\cu\wt{\mf{h}}_1^*$, hence $m\in\wt{\mf{h}}_1^*$; permuting the indices, we obtain $m\in C$. In particular, denoting by $\pi_C$ the gate projection $\overline X\ra C$, we have $m=\pi_Cm(\xi_1,\xi_2,\xi_3)=m(\pi_C\xi_1,\pi_C\xi_2,\pi_C\xi_3)$.

The closures of $\wt{\mf{h}}_i^*$ and $\wt{\mf{k}}_i$ in $\overline X$ are strongly separated by Lemma~\ref{ss vs closures}; let $\{x_i\}$ be the shore of $\overline{\mf{h}_i^*}$ and set $z:=m(x_1,x_2,x_3)$. By Corollary~\ref{strongly separated closed convex sets}, we have $\pi_C(\xi_i)=x_i$, hence $m=z$ no matter which points $\xi_i\in\wt{\mf{k}_i}$ we have chosen.
\end{proof}

In addition to Theorem~\ref{A} and Shalom's superrigidity \cite{Shalom}, the following is a key ingredient in the proof of Theorem~\ref{B}. Since $X$ lacks a cellular structure in general, we are forced to concoct a more elaborate proof for Lemma~\ref{main lemma for SR}, if compared to analogues in the context of trees (pp.\,44-45 in \cite{Shalom}) and cube complexes (Claim~6.1 in \cite{CFI}).

\begin{lem}\label{main lemma for SR}
Suppose that $X$ is irreducible and assume that $\G\acts X$ is Roller nonelementary and Roller minimal. Given $0\neq f\in L^2\left(\mscr{H},\wh{\nu}\right)$, consider 
\[\mscr{S}(f):=\left\{(g_n)\in\G^{\N}\mid g_ng\cdot f\xrightarrow[n\ra+\infty]{} g\cdot f,~\forall g\in \G\right\}.\]
There exists $z\in X$ such that, for every $(g_n)\in\mscr{S}(f)$, there exists $N\geq 0$ such that $g_n\cdot z=z$ for all $n\geq N$.
\end{lem}
\begin{proof}
By part~1 of Proposition~\ref{products}, the barycentric subdivision $X'$ of $X$ is an irreducible median space of the same rank. The action $\G\acts X'$ is without wall inversions, Roller nonelementary and Roller minimal by Lemma~\ref{RM for X'}. As usual, we write $(\mscr{H}',\nu')$ for $(\mscr{H}(X'),\wh{\nu}_{X'})$.

The function $f\in L^2(\mscr{H},\wh{\nu})$ induces $f'\in L^2(\mscr{H}',\nu')$ with $\mscr{S}(f)\cu\mscr{S}(f')$. We approximate $f'$ by a linear combination $F$ of characteristic functions of halfspace intervals $\mscr{H}'(x_1|y_1),...,\mscr{H}'(x_k|y_k)$ such that ${\|F-f'\|<\frac{1}{3}\|f'\|}$. 

Proposition~\ref{double skewering} implies that $\nu'(\mscr{H}')=+\infty$; all halfspace intervals have finite measure, so there exists $\mf{h}'\in\mscr{H}'$ such that $x_i,y_i\in\mf{h}'$ for all $i$. Propositions~\ref{double skewering} and~\ref{strong separation} provide a thick halfspace $\mf{h}\in\mscr{H}'$ such that $\mf{h}^*$ and $\mf{h}'$ are strongly separated; in particular, $\mf{h}$ contains every wall in the set $\mscr{W}'(x_1|y_1)\cup ... \cup\mscr{W}'(x_k|y_k)$.

Propositions~\ref{double skewering} and~\ref{strong separation} also provide $\gamma_1\in\G$ such that $\mf{h}$ and $\gamma_1\mf{h}$ are strongly separated and $\gamma_2\in\G$ such that $\gamma_1\mf{h}^*$ and $\gamma_2\mf{h}$ are strongly separated. We can assume without loss of generality that $d(\gamma_1\mf{h}^*,\gamma_2\mf{h})\geq 1$, as by Proposition~\ref{trivial chains of halfspaces} the quantity $d(\gamma_1\mf{h}^*,\gamma_2^n\mf{h})$ diverges as $n$ goes to infinity. Thus, we can choose elements $\gamma_m\in\G$ such that $\mf{h}^*\supsetneq\gamma_1\mf{h}\supsetneq\gamma_2\mf{h}\supsetneq ...$ and, for all $i\geq 1$, the halfspaces $\gamma_i\mf{h}^*$ and $\gamma_{i+1}\mf{h}$ are strongly separated and at distance at least $1$.

We denote by $\mc{S}\cu\mscr{H}'$ the support of the function $F$ and by $\mc{P}$ the set of all points $x_i,y_i$. Let $D$ be the maximum distance from $\mf{h}^*$ of a point of $\mc{P}$ and $d:=\left\lceil{D+2}\right\rceil$. Let us fix an integer $m>d$ and $g\in\G$ such that $\|g\gamma_mf'-\gamma_mf'\|<\frac{1}{3}\|f'\|$ and $\|gf'-f'\|<\frac{1}{3}\|f'\|$. We prove that $g\gamma_m\mf{h}\cu\gamma_{m-d}\mf{h}$. 

A straightforward repeated application of the triangle inequality yields $\|gF-F\|<2\|F\|$ and $\|g\gamma_mF-\gamma_mF\|<2\|F\|$; thus, $\wh{\nu}(\mc{S}\cap g\mc{S})>0$ and $\wh{\nu}(\gamma_m\mc{S}\cap g\gamma_m\mc{S})>0$. Let $\mf{w}$ be a wall corresponding to a halfspace in $\gamma_m\mc{S}\cap g\gamma_m\mc{S}$. Since $\mf{w}$ is contained in both $\gamma_m\mf{h}$ and $g\gamma_m\mf{h}$ and $\G\acts X'$ is without wall inversions, we conclude that $\gamma_m\mf{h}\cap g\gamma_m\mf{h}\neq\emptyset$. A similar argument shows that $\mf{h}\cap g\mf{h}\neq\emptyset$. 

Let $\mf{u}$ be the wall corresponding to $g\gamma_m\mf{h}$. If $\mf{u}$ is contained in $\gamma_{m-1}\mf{h}$, we either have $g\gamma_m\mf{h}\cu\gamma_{m-1}\mf{h}$ or $\gamma_{m-1}\mf{h}^*\cap g\gamma_m\mf{h}^*=\emptyset$. The former case immediately yields $g\gamma_m\mf{h}\cu\gamma_{m-d}\mf{h}$, while the latter leads to a contradiction as ${\mf{h}\cu\gamma_{m-1}\mf{h}^*}$ and $g\mf{h}\cu g\gamma_m\mf{h}^*$ intersect.

If instead $\mf{u}$ is not contained in $\gamma_{m-1}\mf{h}$, it is contained in $\gamma_m\mf{h}^*$, by strong separation. Let $1\leq l\leq m$ be minimum such that $\gamma_l\mf{h}^*$ contains $\mf{u}$. We have $\gamma_l\mf{h}\cu g\gamma_m\mf{h}$, since $\gamma_l\mf{h}\cap g\gamma_m\mf{h}\supseteq\gamma_m\mf{h}\cap g\gamma_m\mf{h}\neq\emptyset$.

Let $\mf{k}$ be the side of $\mf{w}$ that is contained in $\gamma_m\mf{h}$. Since either $\mf{k}$ or $\mf{k}^*$ lies in $g\gamma_m\mc{S}$, there exists $q\in\mc{P}$ such that $g\gamma_mq\in\mf{k}\cu\gamma_m\mf{h}$. Hence,
\[m-l\leq d(\gamma_m\mf{h},\gamma_l\mf{h}^*)\leq d(g\gamma_mq,g\gamma_m\mf{h}^*)=d(q,\mf{h}^*)\leq D\]
and $m-l+2\leq D+2\leq d$. By strong separation and minimality of $l$, the wall $\mf{u}$ is contained in $\gamma_{l-2}\mf{h}\cu\gamma_{m-d}\mf{h}$. Hence $g\gamma_m\mf{h}\cu\gamma_{m-d}\mf{h}$, since otherwise $g\gamma_m\mf{h}^*\cap\gamma_{m-d}\mf{h}^*=\emptyset$ would again violate the fact that $\mf{h}\cap g\mf{h}\neq\emptyset$.

Now consider the intersection $C$ of the closures in $\overline{X'}$ of the halfspaces $\gamma_m\wt{\mf{h}}$. It consists of a single point $\xi$ since any $\mf{j}\in\mscr{H}'$ with $\wt{\mf{j}}\cap C\neq\emptyset$ and $\wt{\mf{j}}^*\cap C\neq\emptyset$ would have to be transverse to almost all $\gamma_m\mf{h}$, violating strong separation. Strong separation also implies that $\xi$ actually lies in $\overline X\cu\overline{X'}$. 

Given $(g_n)\in\mscr{S}(f')$ we can assume, removing a finite number of elements if necessary, that $\|g_nf'-f'\|<\frac{1}{3}\|f'\|$ for all $n$. Let $N(m)$ be a natural number such that $\|g_n\gamma_mf' -\gamma_mf'\|<\frac{1}{3}\|f'\|$ for all ${n\geq N(m)}$. When ${n\geq N(m)}$, we have shown that $g_n\gamma_m\mf{h}\cu\gamma_{m-d}\mf{h}$; thus we have ${g_n\xi\in g_n\gamma_m\wt{\mf{h}}\cu\gamma_{m-d}\wt{\mf{h}}}$. In this case, strong separation implies that $\s_{g_n\xi}\triangle\s_{\xi}$ consists only of halfspaces whose corresponding walls are contained in $\gamma_{m-d-1}\mf{h}$. This shows that $\limsup\s_{g_n\xi}\triangle\s_{\xi}=\emptyset$; we conclude that $g_n\xi\ra\xi$ in the topology of $\overline X$, for every $(g_n)\in\mscr{S}(f')$.

We finally construct the point $z\in X$. Let $\mf{j},\mf{m}$ be thick halfspaces of $X'$ so that $\mf{m}^*$ and $\mf{j}$ are strongly separated and $\xi\in\wt{\mf{j}}$. Part~2 of Proposition~\ref{strong separation} provides a facing triple consisting of $\mf{m},\mf{m}_1,\mf{m}_2$. We choose thick halfspaces $\mf{j}_1,\mf{j}_2\in\mscr{H}'$ such that $\mf{m}_i^*$ and $\mf{j}_i$ are strongly separated for $i=1,2$; by Proposition~\ref{double skewering} we can find $h_i\in\G$ such that $h_i\mf{j}\cu\mf{j}_i$. Let $z\in X'$ be the point provided by Lemma~\ref{strongly separated closed convex sets 2} applied to $\mf{j},\mf{j}_1,\mf{j}_2$ and $\mf{m},\mf{m}_1,\mf{m}_2$; in particular, we have $z=m(\xi,h_1\xi,h_2\xi)$, hence $z\in X$.

Since the set $\mscr{S}(f')$ is closed under conjugation by elements of $\G$, we have ${g_nh_i\xi\ra h_i\xi}$ for all ${(g_n)\in\mscr{S}(f')}$. Hence, given ${(g_n)\in\mscr{S}(f')}$, there exists $N\geq 0$ such that for every $n\geq N$ we have $g_n\xi\in\wt{\mf{j}}$ and $g_nh_i\xi\in\wt{\mf{j}}_i$, for $i=1,2$. Thus $g_nz=g_nm(\xi,h_1\xi,h_2\xi)=m(g_n\xi,g_nh_1\xi,g_nh_2\xi)=z$.
\end{proof}

In the rest of the section, we consider a locally compact group $G$ and a lattice $\G<G$. Any Borel fundamental domain $U\cu G$ defines a cocycle $\alpha\colon G\x U\ra\G$ such that $gu\in\alpha(g,u)\cdot U$. We say that $\G$ is \emph{square-integrable} if $\G$ is finitely generated and $U$ can be chosen so that
\[\int_U \left|\alpha(g,u)\right|_S^2~du<+\infty,~\forall g\in G;\]
here $du$ is the Haar measure on $U$ and $|\cdot|_S$ denotes the word length with respect to a finite generating set $S\cu\G$. Integrability does not depend on the choice of $S$. Uniform lattices are always square-integrable and a few nonuniform examples were mentioned in the introduction; see \cite{Shalom,Remy99,Remy05,Caprace-Remy09,Caprace-Remy10} for more details and examples. 

We assume that $G$ splits as a product $G_1\x ... \x G_{\ell}$, where each $G_i$ is compactly generated and $\ell\geq 2$. We also require the lattice $\G<G$ to be \emph{irreducible}, i.e.\,to project densely into each factor $G_i$.

Consider a unitary representation ${\pi\colon\G\ra\mscr{U}(\mc{H})}$; we denote by $\mc{H}^0$ the subspace of invariant vectors. Let $c\colon\G\ra\mc{H}$ be a $1$-cocycle for $\pi$. We will make use of the following result of Y. Shalom in an essential way; see page~14 and Theorem~4.1 in \cite{Shalom} for a proof.

\begin{thm}\label{ShalomTHM}
Suppose that $\G$ is square-integrable and that $\mc{H}^0=\{0\}$. There exist $\G$-invariant closed subspaces $\mc{H}_i\leq\mc{H}$, $i=1,...,\ell$, where the restriction $\pi_i\colon\G\ra\mscr{U}(\mc{H}_i)$ extends to a continuous representation $\overline{\pi_i}\colon G\ra\mscr{U}(\mc{H}_i)$ that factors through the projection $p_i\colon G\ra G_i$. Furthermore, there are cocycles $c_i\colon\G\ra\mc{H}_i$, $i=1,...,\ell$, such that $c$ and $c_1+...+c_{\ell}$ represent the same class in $\overline{H^1}(\G,\pi)$.
\end{thm}

The following is a version of our Theorem~\ref{B} under stronger hypotheses.

\begin{thm}\label{SR 1}
Suppose that $\G$ is square-integrable and $X$ is irreducible; let $\G\acts X$ be Roller nonelementary and Roller minimal. There exists a $\G$-invariant, closed median subalgebra $Y\cu X$ where $\G\acts Y$ extends to a continuous action $G\acts Y$. Moreover, $G\acts Y$ factors through a projection $p_i\colon G\ra G_i$.
\end{thm}
\begin{proof}
We have $\overline{H^1}(\G,\rho)\neq\{0\}$ by Theorem~\ref{A}. Lem\-ma~\ref{invariant sets} implies that $\rho$ has no nonzero invariant vectors; thus, Theorem~\ref{ShalomTHM} provides a $\G$-invariant subspace $\{0\}\neq\mc{H}_i\cu L^2(\mscr{H},\wh{\nu})$ where the action of $\G$ extends to a continuous action of $G$ factoring through a projection $p_i\colon G\ra G_i$. 

Pick any $0\neq f\in\mc{H}_i$ and consider the set $\mscr{S}(f)$ introduced in Lemma~\ref{main lemma for SR}. Any sequence $(g_n)\in\G^{\N}$ with $p_i(g_n)\ra\id$ lies in $\mscr{S}(f)$. Thus, Lemma~\ref{main lemma for SR} implies that the $\G$-invariant set
\[Y:=\left\{x\in X\mid\forall (g_n)\in\G^{\N}\text{ s.t. }p_i(g_n)\ra\id, \text{ we have } g_nx\ra x\right\}\]
is nonempty. Note that $Y$ is a median subalgebra of $X$, thus the restriction of the metric of $X$ gives $Y$ a structure of median space. Since $Y$ is a closed subset of $X$, it is a complete median space. Finally, Proposition~4.3 in \cite{Shalom} provides a continuous extension to $G$ of $\G\acts Y$ and this factors through the projection $p_i$. 
\end{proof}

The assumption that $\G\acts X$ be Roller minimal and Roller nonelementary can be replaced with the (stronger) requirement that $\G$ have no finite orbit in the visual boundary of the {\rm CAT}(0) space $\wh{X}$; see Proposition~5.2 in \cite{Fioravanti2}.

The homomorphism $\phi\colon G\ra\text{Isom~Y}$ provided by Theorem~\ref{SR 1} is continuous with respect to the topology of pointwise convergence. We remark however that $\phi$ remains continuous even if we endow $\text{Isom}~Y$ with the topology mentioned in Remark~\ref{the other topology} below; this will be a key point in our proof of Theorem~\ref{C}.

\begin{rmk}\label{the other topology}
In the proof of Theorem~\ref{SR 1}, Lemma~\ref{main lemma for SR} actually yields that the smaller set
\[Y_0:=\left\{x\in X\mid\forall (g_n)\in\G^{\N}\text{ s.t. }p_i(g_n)\ra\id, \exists N\geq 0\text{ s.t. } g_nx=x, \forall n\geq N\right\}\]
is nonempty. Thus, $\phi\colon G\ra\text{Isom~Y}$ is continuous with respect to the topology on $\text{Isom}~Y$ that is generated by stabilisers of points of $Y_0$. In the statement of Theorem~\ref{SR 1}, we can always take $Y$ to be the closure of $Y_0$ in $X$.

This topology on $\text{Isom}~Y$ might seem a lot finer than the topology of pointwise convergence. To clarify this phenomenon, we mention the following fact, without proof. Let $W$ be an irreducible, complete, finite rank median space admitting a Roller nonelementary, Roller minimal action; then there exists a dense, convex subset $C\cu W$ such that, for every $x\in C$, the stabiliser of $x$ is open for the topology of pointwise convergence on $\text{Isom}~W$. It is not hard to derive this from Lemma~\ref{strongly separated closed convex sets 2}.
\end{rmk}

Relaxing the hypotheses of Theorem~\ref{SR 1}, we obtain Theorem~\ref{B} for all square-integrable lattices:

\begin{cor}\label{SR 2}
Suppose that $\G$ is square-integrable; let $\G\acts X$ be Roller nonelementary. There exist a finite index subgroup $\G_0\leq\G$, a $\G_0$-invariant component $Z\cu\overline X$ and a $\G_0$-invariant closed median subalgebra $Y\cu Z$ where the action $\G_0\acts Y$ extends to a continuous action $G_0\acts Y$, for an open finite index subgroup $G_0\leq G$.
\end{cor}
\begin{proof}
We proceed by induction on $\text{rank}(X)$; when the rank is zero there is nothing to prove, so we assume that the statement holds for all median spaces of rank at most $r-1$. By Proposition~\ref{Roller elementary vs strongly so}, there exists a $\G$-invariant, closed, convex subset $D$ of a component $W\cu\overline X$ such that $\G\acts D$ is Roller minimal and Roller nonelementary. If $W\cu\partial X$, we have $\text{rank}(D)<r$ and we conclude by the inductive hypothesis; thus we can assume that $W=X$.

Let $D=D_1\x ... \x D_k$ be the splitting of $D$ into irreducible factors provided by Proposition~\ref{products}. If $k=1$, the result follows from Theorem~\ref{SR 1}. If $k\geq 2$, let $\G_1\leq\G$ be a finite index subgroup preserving the splitting of $D$; up to permuting the factors, we can assume that $\G_1\acts D_i$ is Roller nonelementary for $1\leq i\leq s$ and Roller elementary for $i>s$. A further finite index subgroup $\G_2\leq\G_1$ fixes a point $\xi_i\in\overline{D_i}$ for each $i>s$; we denote by $Z_i\cu\overline{D_i}$ the component containing $\xi_i$. Note that $\G_2$ is a square-integrable, irreducible lattice in an open, finite index subgroup of $G$.

Since $\text{rank}(D_i)<r$, for each $i\leq s$ the inductive hypothesis yields a finite index subgroup $\G_{0i}\leq\G_2$, an open finite index subgroup $G_{0i}\leq G$ and a $\G_{0i}$-invariant, closed median subalgebra $Y_i$ of a component $Z_i\cu\overline{D_i}$ where the action of $\G_{0i}$ extends to a continuous action of $G_{0i}$. Let $\G_0$ be the intersection of all $\G_{0i}$ and $G_0$ be the intersection of all $G_{0i}$, for $i\leq s$. The set $Y:=Y_1\x ... \x Y_s\x\{\xi_{s+1}\}\x ... \x\{\xi_k\}\cu\overline D$ is a closed median subalgebra of $Z_1\x ... \x Z_k$, which is a component of $\overline D$; in particular, $Z_1\x ... \x Z_k$ is a closed, convex subset of a component $Z\cu\overline X$. The action $\G_0\acts Y$ trivially extends to a continuous action $G_0\acts Y$.
\end{proof}

We now describe two examples that illustrate how:
\begin{itemize}
\item in Theorem~\ref{SR 1} the space $Y$ cannot be taken to coincide with $X$, nor with a convex subset (Example~\ref{need Y});
\item in Corollary~\ref{SR 2} it cannot be avoided to pass to the finite index subgroup $\G_0$, even when the action is Roller minimal (Example~\ref{need G_0}).
\end{itemize}
The actions that we consider are actually on {\rm CAT}(0) square complexes. Since Burger-Mozes groups play an important role in the construction of the two examples, we briefly recall a few facts regarding their construction.

Given an integer $n\geq 3$, we denote by $T_n$ the $n$-regular tree and by $A_n$ the group of even permutations on $n$ elements. We fix a \emph{legal colouring} on $T_n$, i.e.\,a way of associating an integer in $\{1,...,n\}$ to every edge of $T_n$ so that we see all $n$ integers around each vertex; in particular, we have a bijection $i_v\colon\text{lk}(v)\ra\{1,...,n\}$ for every vertex $v$. Let $U(A_n)\leq\text{Isom}~T_n$ be the subgroup of isometries $g$ such that $i_{gv}\o g\o i_v^{-1}\in A_n$ for every vertex $v$ of $T_n$; we denote by $U(A_n)^+$ the intersection of $U(A_n)$ with the subgroup of $\text{Isom}~T_n$ generated by edge stabilisers. If $n\geq 4$, the subgroup $U(A_n)^+$ has index $2$ in $U(A_n)$, see Proposition~3.2.1 in \cite{Burger-Mozes(A)}.

The subgroup $U(A_n)$ is closed in $\text{Isom}~T_n$; in particular, it is locally compact, second countable and compactly generated (e.g.\,by Theorem~4.C.5 and Proposition~5.B.5 in \cite{Cornulier-DeLaHarpe}). By Theorem~6.3 in \cite{Burger-Mozes(B)}, there exists a uniform irreducible lattice $\Lambda\leq U(A_{2k})\x U(A_{2k})$ for every integer $k\geq 19$.

For the next two examples, we fix such a lattice $\Lambda$. Let ${p_1,p_2\colon\Lambda\ra U(A_{2k})}$ be the projections into the two factors and set $\Lambda_0:=p_1^{-1}(U(A_{2k})^+)$; this is an irreducible lattice in the open, index 2 subgroup $U(A_{2k})^+\x U(A_{2k})$ of $U(A_{2k})\x U(A_{2k})$. Let $\tau\colon\Lambda\ra\Z/2\Z$ be the homomorphism with kernel $\Lambda_0$.

\begin{ex}\label{need Y}
Given any tree $T$, we can blow up every edge to a square as in Figure~\ref{tree of squares}, thus obtaining a ``tree of squares'' $\mathds{T}$. Adjacent squares only share a vertex; if $T$ has no leaves, each square has a pair of opposite vertices that are shared with other squares and a pair of opposite vertices that are not shared. The space $\mathds{T}$ is a complete, rank two median space in which $T$ embeds as a median subalgebra; edges of $T$ correspond to diagonals joining shared pairs of vertices of squares of $\mathds{T}$.

We can embed $\text{Isom}~T\hookrightarrow\text{Isom}~\mathds{T}$ by extending each isometry of $T$ so that the restriction to each square is orientation preserving. Let $\s\in\text{Isom}~\mathds{T}$ be the isometry that fixes pointwise the image of the embedding $T\hookrightarrow\mathds{T}$ and acts on each square as a reflection in the diagonal; we have $\s^2=\id$ and ${\text{Isom}~T\x\langle\s\rangle\hookrightarrow\text{Isom}~\mathds{T}}$. 

Let us now apply this construction to $T=T_{2k}$, writing $\mathds{T}=\mathds{T}_{2k}$. Let $\Lambda\leq U(A_{2k})\x U(A_{2k})$ be as above. Viewing $p_2\colon\Lambda\ra U(A_{2k})$ as a homomorphism into $\text{Isom}~T_{2k}$ we can define a homomorphism $\Lambda\ra\text{Isom}~T_{2k}\x\langle\s\rangle$ by $\lambda\mapsto(p_2(\lambda),\tau(\lambda))$. We denote by $\psi$ the composition of this map with the embedding $\text{Isom}~T_{2k}\x\langle\s\rangle\hookrightarrow\text{Isom}~\mathds{T}_{2k}$.

The action $\Lambda\acts\mathds{T}_{2k}$ induced by $\psi$ is Roller nonelementary and Roller minimal since the action $\Lambda\acts T_{2k}$ induced by $p_2$ is. As $\mathds{T}_{2k}$ is irreducible, Theorem~\ref{SR 1} guarantees a continuous extension of $\Lambda\acts Y$ to $U(A_{2k})$, for some $\Lambda$-invariant median subalgebra of $\mathds{T}_{2k}$. Indeed, one can take $Y$ to be the image of $T_{2k}\hookrightarrow\mathds{T}_{2k}$. 

However, $Y$ cannot be taken to be a convex subspace (or even a subcomplex) of $\mathds{T}_{2k}$. Indeed, $Y$ would be forced to be the whole $\mathds{T}_{2k}$, as this is the only $\Lambda$-invariant convex subset of $\mathds{T}_{2k}$. The action $\Lambda\acts\mathds{T}_{2k}$ does not extend to $U(A_{2k})\x U(A_{2k})$ by factoring via $p_1$; this is because, whenever elements $g_n\in\Lambda$ satisfy $p_1(g_n)\ra\id$, the sequence $(p_2(g_n))$ must diverge. However, $\Lambda\acts\mathds{T}_{2k}$ also does not extend by factoring through $p_2$: we have $\overline{p_2(\Lambda_0)}=\overline{p_2(\Lambda)}=U(A_{2k})$, but $\psi(\Lambda_0)$ is contained in the closed subgroup $\text{Isom}~T_{2k}<\text{Isom}~\mathds{T}_{2k}$ and $\psi(\Lambda)$ is not.
\end{ex}

\begin{figure}
\centering
\includegraphics[width=4.5in]{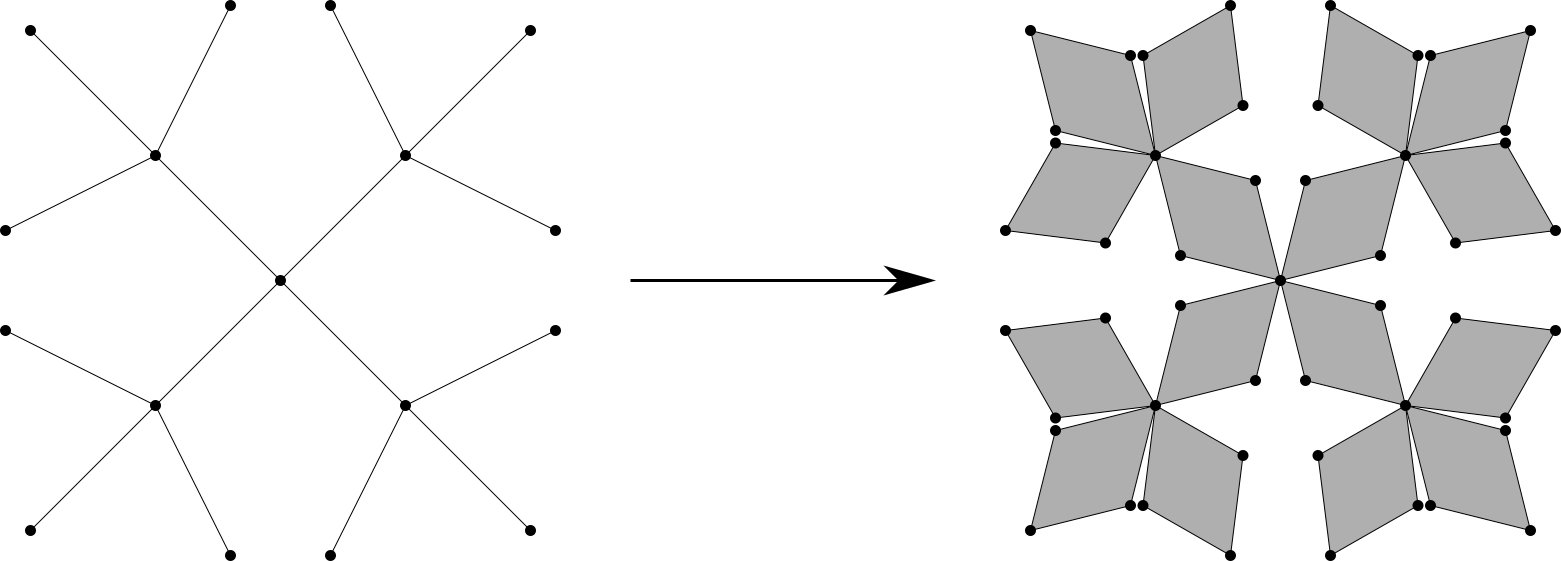}
\caption{}
\label{tree of squares}
\end{figure}

In the next example we maintain the notation introduced before Example~\ref{need Y}.

\begin{ex}\label{need G_0} 
Choose an element $g\in\Lambda\setminus\Lambda_0$ and consider the action ${\Lambda_0\acts T_{2k}\x T_{2k}}$ given by $\lambda\cdot(x,y)=(p_2(\lambda)\cdot x, p_2(g^{-1}\lambda g)\cdot y)$. Since the action $U(A_{2k})\acts T_{2k}$ does not preserve any proper closed subtree, the same holds for the action of $p_2(\Lambda_0)$. Part~3 of Proposition~\ref{products} then implies that $\Lambda_0$ does not leave any proper, closed, convex subset of $T_{2k}\x T_{2k}$ invariant. Note that no component of $\partial(T_{2k}\x T_{2k})=(\partial T_{2k}\x\overline{T_{2k}})\cup(\overline{T_{2k}}\x\partial T_{2k})$ is preserved by $\Lambda_0$, as this would correspond to a fixed point for $p_2(\Lambda_0)\acts\overline{T_{2k}}$, hence to a fixed point for $U(A_{2k})\acts\overline{T_{2k}}$. We conclude that $\Lambda_0\acts T_{2k}\x T_{2k}$ is Roller minimal and the same argument also shows that it is Roller nonelementary.

One can easily check that $\Lambda_0\acts T_{2k}\x T_{2k}$ can be extended to an action of the whole $\Lambda$ by setting $\lambda g\cdot(x,y)=(p_2(\lambda g^2)\cdot y, p_2(g^{-1}\lambda g)\cdot x)$ for all $\lambda\in\Lambda_0$. This action also is Roller minimal and Roller nonelementary. We will show, however, that there exists no $\Lambda$-equivariant isometric embedding $j\colon Y\hookrightarrow\overline{T_{2k}}\x\overline{T_{2k}}$ of a median space $Y$ such that the action on $Y$ extends continuously to $U(A_{2k})\x U(A_{2k})$ by factoring through one of the factors.

Let $j\colon Y\hookrightarrow\overline{T_{2k}}\x\overline{T_{2k}}$ be a $\Lambda$-equivariant embedding; note that $j(Y)$ is entirely contained in a $\Lambda$-invariant component ${Z\cu\overline{T_{2k}}\x\overline{T_{2k}}}$. In particular, the previous discussion shows that $Z=T_{2k}\x T_{2k}$.

By Lemma~6.5 in \cite{Bow1}, each wall of $j(Y)$ arises from a wall of $T_{2k}\x T_{2k}$, i.e. a wall of one of the two factors, see part~1 of Proposition~\ref{products}. Since the two factors are exchanged by $g\in\Lambda$, we conclude that $Y$ splits as $Y_1\x Y_2$, with $\Lambda_0\acts Y$ preserving this decomposition and $g$ exchanging $Y_1$ and $Y_2$. 

Suppose for the sake of contradiction that $\Lambda\acts Y$ extends to an action of $U(A_{2k})\x U(A_{2k})$ by factoring through one of the two factors. As in Example~\ref{need Y}, we see that the extension cannot factor via $p_1$. However, since $p_2(\Lambda_0)$ is dense in $U(A_{2k})$ and $\Lambda_0$ preserves the splitting $Y=Y_1\x Y_2$, part~2 of Proposition~\ref{products} implies that an extension factoring through $p_2$ would also preserve the splitting $Y=Y_1\x Y_2$. This contradicts the fact that $g$ exchanges $Y_1$ and $Y_2$.
\end{ex}

We conclude the section by proving Theorem~\ref{C}.

\begin{proof}[Proof of Theorem~\ref{C}]
We begin by observing that part~2 follows from part~1 and Proposition~\ref{chi_xi continuous}. Now, suppose for the sake of contradiction that $\G$ admits a Roller nonelementary action on $X$. As in the proof of Theorem~\ref{Z}, we can assume that $X$ is irreducible and that $\G\acts X$ is Roller minimal. Theorem~\ref{SR 1} then yields a factor $G_i$, a closed median subalgebra $Y\cu X$ and actions $G_i\acts Y$ and $\G\acts Y$. Without loss of generality, we can assume that $Y$ is the closure of $Y_0$ inside $X$, as in Remark~\ref{the other topology}. 

Stabilisers of points of $Y_0$ are open in $G_i$, thus the identity component $G^0_i$ must fix $Y_0$ pointwise. As $Y_0$ is dense in $Y$, the entire action $G^0_i\acts Y$ vanishes and $G_i\acts Y$ descends to an action of the group $G_i/G^0_i$. Since $G_i$ satisfies condition $(*)$, Theorem~\ref{Z} above and Corollary~6.5 in \cite{Fioravanti2} imply that the action $G_i/G^0_i\acts Y$ is Roller elementary. However, by Lemma~\ref{Roller of subalgebras}, the actions $\G\acts Y$ and $G_i\acts Y$ are Roller nonelementary, a contradiction. 
\end{proof}

\subsection{Homomorphisms to coarse median groups.}

We defined equivariantly coarse median groups in the introduction. Here we simply prove Corollary~\ref{E}. 

\begin{proof}[Proof of Corollary~\ref{E}]
Fix a non-principal ultrafilter $\om$ on $\N$ and let $H_{\om}$ be the corresponding ultrapower of $H$. We endow $H$ with a word metric $d_S$ arising from a finite generating set $S\cu H$. Given $\underline\lambda=(\lambda_n)\in\R_+^{\N}$, we denote by $\text{Cone}_{\om}(H,\underline\lambda)$ the asymptotic cone obtained by taking all basepoints at the identity and $\underline\lambda$ as sequence of scaling factors. Let $d$ denote the metric that $d_S$ induces on $\text{Cone}_{\om}(H,\underline\lambda)$; it is a geodesic metric and it is preserved by the natural action $H_{\om}\acts\text{Cone}_{\om}(H,\underline\lambda)$.

If $\lambda_n\ra+\infty$, the coarse median on $H$ induces a structure of finite rank median algebra on $\text{Cone}_{\om}(H,\underline\lambda)$, see Section~9 in \cite{Bow1}; we denote by $m$ the corresponding median map. The action $H_{\om}\acts\text{Cone}_{\om}(H,\underline\lambda)$ is by automorphisms of the median algebra structure. By Propositions~3.3 and~5.1 in \cite{Zeidler}, we can endow $\text{Cone}_{\om}(H,\underline\lambda)$ with a median metric $d_m$ that is bi-Lipschitz equivalent to $d$ and preserved by the $H_{\om}$-action; furthermore, the median algebra structure associated to $d_m$ is given by the map $m$. 

Now suppose for the sake of contradiction that there exist pairwise non-conjugate homomorphisms $\phi_n\colon\G\ra H$, for $n\geq 0$; these correspond to a homomorphism $\G\ra H_{\om}$, hence to an action on every asymptotic cone of $H$ that preserves the median metric $d_m$. The Bestvina-Paulin construction \cite{Bestvina,Paulin,Paulin-Arb} provides us with a sequence $\mu_n\ra+\infty$ such that, modifying each $\phi_n$ within its conjugacy class if necessary, the induced action $\G\acts\text{Cone}_{\om}(H,\underline\mu)$ has no global fixed point. This, however, contradicts Theorem~\ref{C}.
\end{proof}



\appendix

\section{Structure of UBS's.}\label{structure of UBS's}

Let $X$ be a complete median space of finite rank $r$. We fix points $x\in X$ and $\xi\in\partial X$. Let $\Om\cu\s_{\xi}\setminus\s_x$ be a minimal, reduced UBS (Definition~\ref{UBS defn}).

\begin{lem}
Let $g\in\text{Isom}_{\xi}X$ be an isometry satisfying $g\Om\sim\Om$. Consider the UBS ${\Om_1:=\bigcap_{0\leq i\leq r}g^{-i}\Om\cu\Om}$.
\begin{enumerate}
\item If $\chi_{\Om}(g)\geq 0$, we have $g^{r!}\Om_1\cu\Om_1$ and $g^{r!}\mf{h}\cu\mf{h}$ for all $\mf{h}\in\Om_1$.
\item If $\chi_{\Om}(g)=0$, we have $g^{r!}\Om_1=\Om_1$ and $g^{r!}\mf{h}=\mf{h}$ for all $\mf{h}\in\Om_1$.
\end{enumerate}
\end{lem}
\begin{proof}
Observe that $\Om_1\sim\Om$; we fix a diverging chain $(\mf{k}_n)_{n\geq 0}$ in $\Om_1$. If $\mf{h}\in\Om_1$, the halfspaces $g^i\mf{h}$ with $0\leq i\leq r$ all lie in $\Om\cu\s_{\xi}\setminus\s_x$ and cannot be pairwise transverse; thus, there exists $0\leq i\leq r$ such that either $g^i\mf{h}\cu\mf{h}$, or $g^i\mf{h}\supseteq\mf{h}$. Hence, for every $\mf{h}\in\Om_1$ we either have $g^{r!}\mf{h}\cu\mf{h}$ or $g^{r!}\mf{h}\supseteq\mf{h}$.  

Suppose that $g^{r!}\mf{h}_0\subsetneq\mf{h}_0$ for some $\mf{h}_0\in\Om_1$; set $\mf{h}_k:=g^{kr!}\mf{h}_0$. For each $k\geq 0$, a cofinite subchain of $\{g^{-kr!}\mf{k}_n\}_{n\geq 0}$ is a diverging chain in $\Om_1$, as $g\Om_1\sim\Om_1$. Hence, for each $k\geq 0$ we have $\mf{h}_0\supseteq g^{-kr!}\mf{k}_n$ if $n$ is sufficiently large, since $\Om$ is reduced; in particular $\mf{h}_0\supseteq\mf{h}_k\supseteq\mf{k}_n$. We conclude that each $\mf{h}_k$ lies in $\Om_1$; Proposition~\ref{trivial chains of halfspaces} guarantees that $(\mf{h}_k)_{k\geq 0}$ is a diverging chain. 

Let $\Xi\cu\Om_1$ be the inseparable closure of $\{\mf{h}_k\}_{k\geq 0}$; it is a UBS equivalent to $\Om$ and it satisfies $g^{r!}\Xi\cu\Xi$. Observe that, for each $k\geq 0$,
\[kr!\cdot\chi_{\Om}(g)=\chi_{\Om}(g^{kr!})=\chi_{\Xi}(g^{kr!})=\wh{\nu}(\Xi\setminus g^{kr!}\Xi)\geq d(\mf{h}_k,\mf{h}_0^*);\]
since $d(\mf{h}_k,\mf{h}_0^*)>0$ for some $k\geq 0$, we conclude that $\chi_{\Om}(g)>0$. The same argument applied to $g^{-1}$ shows that $\chi_{\Om}(g)<0$ if there exists $\mf{h}_0\in\Om_1$ with $g^{r!}\mf{h}_0\supsetneq\mf{h}_0$.

This proves part~2 and shows that $g^{r!}\mf{h}\cu\mf{h}$ for all $\mf{h}\in\Om_1$ if $\chi_{\Om}(g)>0$. In the latter case, for every $\mf{h}\in\Om_1$ we have $\mf{h}\supseteq g^{r!}\mf{h}\supseteq\mf{k}_n$ for sufficiently large $n$, since $\Om$ is reduced; thus, $g^{r!}\mf{h}\in\Om_1$ and $g^{r!}\Om_1\cu\Om_1$.
\end{proof}
 
\begin{cor}\label{gXi_2}
Let $g\in\text{Isom}_{\xi}X$ be an isometry satisfying $g\Om\sim\Om$. Define $\Om_1$ as in the previous lemma and consider the UBS ${\Om_2:=\bigcap_{1\leq i\leq r!}g^{i-1}\Om_1}$.
\begin{enumerate}
\item If $\chi_{\Om}(g)\geq 0$, we have $g\Om_2\cu\Om_2$.
\item If $\chi_{\Om}(g)=0$, we have $g\Om_2=\Om_2$.
\end{enumerate}
\end{cor}

In the rest of the appendix, we also consider a compact subset ${K\cu\text{Isom}_{\xi}X}$ such that $g\Om\sim\Om$ for every $g\in K$. 

\begin{lem}\label{C_1}
\begin{enumerate}
\item There exists a constant $C_1=C_1(\Om,K)$ such that every $\mf{h}\in\Om$ with $d(x,\mf{h})>C_1$ lies in $g\Om$ for every $g\in K$.
\item If $\Xi\cu\s_{\xi}\setminus\s_x$ is a minimal, reduced UBS such that $\Xi\not\sim\Om$ and $g\Xi\sim\Xi$ for all $g\in K$, there exists a UBS $\overline{\Om}\cu\Om$ that is disjoint from $g\Xi$ for all $g\in K$.
\end{enumerate}
\end{lem}
\begin{proof}
Let $(\mf{h}_n)_{n\geq 0}$ be a diverging chain in $\Om$ with $\mf{h}_0$ thick. For every $g\in K$, a cofinite subchain of $(g^{-1}\mf{h}_n)_{n\geq 0}$ is contained in $\Om$, as $g\Om\sim\Om$; since $\Om$ is reduced, there exists $n(g)\geq 0$ so that $g^{-1}\mf{h}_{n(g)}\cu\mf{h}_0$. By Proposition~\ref{trivial chains of halfspaces} we can assume that $\wh{\nu}(\mscr{H}(g\mf{h}_0^*|\mf{h}_{n(g)}))>0$ and Lemma~\ref{aux lemma 1} provides a neighbourhood $U(g)$ of $g$ in $K$ such that $\mscr{H}(\gamma\mf{h}_0^*|\xi)\cap\mscr{H}(g\mf{h}_0^*|\mf{h}_{n(g)})\neq\emptyset$ for all $\gamma\in U(g)$; in particular, $\mf{h}_{n(g)}\cu\gamma\mf{h}_0$ for all $\gamma\in U(g)$. There exist $g_1,...,g_k\in K$ such that $K=U(g_1)\cup ... \cup U(g_k)$; if $N$ is the maximum of the $n(g_i)$, we have $\mf{h}_N\cu g\mf{h}_0$ for all $g\in K$. Let $\Om_N$ be the inseparable closure of $\{\mf{h}_n\}_{n\geq N}$. If $m\geq N$ and $g\in K$, we have $g\mf{h}_0\supseteq\mf{h}_m\supseteq g\mf{h}_n$ for every sufficiently large $n$, since $\Om$ is reduced. This shows that $\Om_N$ is contained in $g\Om$ for all $g\in K$. Since $\Om_N\sim\Om$, there exists a constant $C_1$ such that every $\mf{h}\in\Om$ with $d(x,\mf{h})>C_1$ lies in $\Om_N$.

To prove part~2, let $C$ be the supremum of distances $d(x,\mf{h})$ for $\mf{h}\in\Om\cap\Xi$; we have $C<+\infty$ since $\Om\not\sim\Xi$. Let $M$ be the maximum distance $d(x,gx)$ for $g\in K$ and consider $C':=\max\{C,C_1(\Xi,K^{-1})+M\}$; we define $\overline\Om$ to be the set of $\mf{h}\in\Om$ with $d(x,\mf{h})>C'$. If there existed $\mf{h}\in\overline\Om\cap g\Xi$ for some $g\in K$, we would have $g^{-1}\mf{h}\in\Xi$ and $d(x,g^{-1}\mf{h})>C_1(\Xi,K^{-1})$; thus, part~1 implies that $g^{-1}\mf{h}\in g^{-1}\Xi$, i.e. $\mf{h}\in\Xi$, and this contradicts the fact that $\mf{h}\in\Om$ and $d(x,\mf{h})>C$.
\end{proof}

Recall that we have introduced (in Section~\ref{Haag main} after Proposition~\ref{point x_S}) the function $\alpha_{\Om}\colon\mscr{H}\ra\R$, defined by the formula $\alpha_{\Om}(\mf{h}):=\wh{\nu}\left(\mscr{H}(x|\mf{h})\cap\Om\right)$ and the sets ${\Om_c:=\{\mf{h}\in\Om\mid\alpha_{\Om}(\mf{h})\leq c\}}$. Observe that $\alpha_{\Om}(\mf{k})\leq\alpha_{\Om}(\mf{h})$ whenever $\mf{h}\cu\mf{k}$; in particular $\alpha_{\Om}$ is measurable. We have $\alpha_{\Om}(\mf{h})\leq\wh{\nu}(\mscr{H}(x|\mf{h}))=d(x,\mf{h})$ for all $\mf{h}\in\mscr{H}$. 

We say that $\Om$ is \emph{small} if $\wh{\nu}(\Om)<+\infty$; otherwise, $\Om$ is \emph{large}. If $X$ is a {\rm CAT}(0) cube complex, every UBS is large; an example of a small UBS in a rank two median space appears in Figure~3 of \cite{Fioravanti1}. If $\Om$ is small, we have $\chi_{\Om}(h)=0$ for every isometry $h$ fixing $[\Om]$. Note that the supremum of $\alpha_{\Om}$ is precisely $\wh{\nu}(\Om)$.

\begin{lem}\label{chains in small UBS's}
Let $\mf{h}_n\in\Om$ be halfspaces with $\mf{h}_{n+1}\cu\mf{h}_n$ for all $n\geq 0$. Then, ${\alpha_{\Om}(\mf{h}_n)\ra\wh{\nu}(\Om)}$ if and only if $(\mf{h}_n)_{n\geq 0}$ is a diverging chain of halfspaces.
\end{lem}
\begin{proof}
The fact that $\alpha_{\Om}(\mf{h}_n)\ra\wh{\nu}(\Om)$ if $(\mf{h}_n)_{n\geq 0}$ is a diverging chain follows from the fact that $\Om$ is reduced. For the other implication, let $(\mf{k}_m)_{m\geq 0}$ be a diverging chain in $\Om$. Since $\alpha_{\Om}(\mf{h})\leq d(x,\mf{h})$, it suffices to consider the case when $\Om$ is small. For every $m\geq 0$, the set $\{\mf{j}\in\Om\mid\mf{j}\cu\mf{k}_m\}$ has measure $a_m>0$ by Proposition~\ref{trivial chains of halfspaces}. For large $n$ we have $\alpha_{\Om}(\mf{h}_n)>\wh{\nu}(\Om)-a_m$, hence there exists $\mf{j}\cu\mf{k}_m$ such that $j\in\mscr{H}(x|\mf{h}_n)$; in particular, $\mf{h}_n\cu\mf{k}_m$. Since $m$ is arbitrary, this shows that $(\mf{h}_n)_{n\geq 0}$ is a diverging chain.
\end{proof}

\begin{lem}\label{properties of Om_c}
\begin{enumerate}
\item For every $0\leq c<\wh{\nu}(\Om)$, the set ${\Om\setminus\Om_c}$ is a UBS.
\item For all $c\geq 0$, we have $\wh{\nu}(\Om_c)\leq rc$.
\end{enumerate}
\end{lem}
\begin{proof}
Since $\Om$ is reduced, any $\mf{h}\in\Om\setminus\Om_c$ contains almost every halfspace in any diverging chain in $\Om$; this provides a diverging chain in $\Om\setminus\Om_c$. Inseparability follows from the monotonicity of $\alpha_{\Om}$.

To prove part~2, we decompose $\Om_c=\mc{C}_1\sqcup ... \sqcup\mc{C}_k$ as in Lemma~\ref{Dilworth for differences}. If $\mf{h},\mf{k}\in\mc{C}_i$ and $\mf{h}\cu\mf{k}$, we have $c\geq\alpha_{\Om}(\mf{h})\geq\wh{\nu}\left(\mscr{H}(\mf{k}^*|\mf{h})\right)$. Hence, Lemma~2.27 in \cite{Fioravanti1} implies that the inseparable closure of $\mc{C}_i$ has measure at most $c$. We conclude that $\wh{\nu}(\Om_c)\leq kc\leq rc$.
\end{proof}

\begin{lem}\label{properties of Om_c 3}
Assume that $g\Om\cu\Om$ for all $g\in K$.
\begin{enumerate}
\item For all $\mf{h}\in\Om$ and $g\in K$, we have
\[-d(x,gx)-\chi_{\Om}(g)\leq\alpha_{\Om}(g^{-1}\mf{h})-\alpha_{\Om}(\mf{h})\leq d(x,gx);\]
in particular $\|g\alpha_{\Om}-\alpha_{\Om}\|_{\infty}\leq C_2$ for some constant $C_2=C_2(\Om,K)$.
\item For every $c>0$, we have $g\Om_c\cu\Om_{c+C_2}$.
\end{enumerate}
\end{lem}
\begin{proof}
Indeed,
\begin{align*}
\alpha_{\Om}(g^{-1}\mf{h})&\leq\wh{\nu}\left(\mscr{H}(x|g^{-1}x)\right)+\wh{\nu}\left(\mscr{H}(g^{-1}x|g^{-1}\mf{h})\cap\Om\right) \\
&= d(x,g^{-1}x)+\wh{\nu}\left(\mscr{H}(x|\mf{h})\cap g\Om\right)\leq d(x,gx)+\alpha_{\Om}(\mf{h}), 
\end{align*}
and
\begin{align*}
\alpha_{\Om}(\mf{h})&\leq\wh{\nu}\left(\mscr{H}(x|gx)\right)+\wh{\nu}\left(\mscr{H}(gx|\mf{h})\cap\Om\right) \\
&= d(x,gx)+\wh{\nu}\left(\mscr{H}(x|g^{-1}\mf{h})\cap g^{-1}\Om\right) \\
&\leq d(x,gx)+\wh{\nu}\left(\mscr{H}(x|g^{-1}\mf{h})\cap\Om\right)+\wh{\nu}\left(g^{-1}\Om\setminus\Om\right) \\
&=d(x,gx)+\alpha_{\Om}(g^{-1}\mf{h})+\chi_{\Om}(g).
\end{align*}
We then take $C_2$ to be the maximum of $d(x,gx)+\chi_{\Om}(g)$ for $g\in K$; this exists due to Proposition~\ref{chi_xi continuous}. 
Regarding part~2, observe that, if $\mf{h}\in\Om_c$,
\[\alpha_{\Om}(g\mf{h})=\alpha_{\Om}(\mf{h})+\left(\alpha_{\Om}(g\mf{h})-g\alpha_{\Om}(g\mf{h})\right)\leq c+C_2.\]
\end{proof}

\begin{lem}\label{strongly red}
Assume that $\Om$ is strongly reduced. 
\begin{enumerate}
\item For every $d\geq 0$, there exists a constant $C_3=C_3(\Om,d)$ such that $\mf{h}\cu\mf{k}$ for all $\mf{h},\mf{k}\in\Om$ with $d(x,\mf{h})>C_3$ and $d(x,\mf{k})\leq d$.
\item If $\Xi\cu\Om$ is a UBS and $d(x,\mf{k})\leq d$ for all $\mf{k}\in\Om\setminus\Xi$, we have $\alpha_{\Om}(\mf{h})-\alpha_{\Xi}(\mf{h})=\wh{\nu}(\Om\setminus\Xi)$ for all $\mf{h}\in\Om$ with $d(x,\mf{h})>C_3(\Om,d)$.
\end{enumerate}
\end{lem}
\begin{proof}
Decompose $\Om=\mc{C}_1\sqcup ... \sqcup\mc{C}_k$, where each $\mc{C}_i$ is totally ordered by inclusion and contains a diverging chain. Pick halfspaces $\mf{k}_i\in\mc{C}_i$ with $d(x,\mf{k}_i)>d$. By part~3 of Lemma~\ref{reduced and strongly reduced}, the halfspace $\mf{k}_i$ cannot be transverse to a diverging chain in $\Om$; thus, part~2 of Lemma~\ref{Dilworth for differences} guarantees that the halfspaces of $\Om$ that are transverse to $\mf{k}_i$ lie at uniformly bounded distance from $x$. We conclude that there exists $C_3$ such that every $\mf{h}\in\Om$ with $d(x,\mf{h})>C_3$ is contained in each $\mf{k}_i$, hence also in each $\mf{k}\in\Om$ with $d(x,\mf{k})\leq d$. This proves part~1; part~2 is an immediate consequence.
\end{proof}

In the rest of the section, we assume that $\Om$ is minimal and strongly reduced.

\begin{lem}\label{properties of Om_c 2}
Suppose that $\chi_{\Om}(g)\geq 0$ for all $g\in K$.
\begin{enumerate}
\item There exists a constant $C_4=C_4(\Om,K)$ such that, for every halfspace $\mf{h}\in\Om$ with $d(x,\mf{h})>C_4$ and every $g\in K$, we have ${\alpha_{\Om}(g\mf{h})\geq\alpha_{\Om}(\mf{h})}$. 
\item There exists a constant $C_5=C_5(\Om,K)<\wh{\nu}(\Om)$ such that we have $g(\Om\setminus\Om_c)\cu\Om\setminus\Om_c$ and $\Om_c\setminus g\Om_c=\Om\setminus g\Om$ whenever $c>C_5$ and $g\in K$.
\end{enumerate}
\end{lem}
\begin{proof}
We first observe that, given a minimal UBS $\Xi\cu\s_{\xi}\setminus\s_x$ and an isometry $g\in\text{Isom}_{\xi}X$ with $g\Xi\cu\Xi$, we have:
\begin{align*}
\alpha_{\Xi}&(g\mf{h})-\alpha_{\Xi}(\mf{h})=\wh{\nu}\left(\mscr{H}(g^{-1}x|\mf{h})\cap g^{-1}\Xi\right)-\wh{\nu}\left(\mscr{H}(x|\mf{h})\cap\Xi\right) \\
&=\wh{\nu}\left(\mscr{H}(g^{-1}x|\mf{h})\cap\Xi\right)-\wh{\nu}\left(\mscr{H}(x|\mf{h})\cap\Xi\right)+\wh{\nu}\left(\mscr{H}(g^{-1}x|\mf{h})\cap(g^{-1}\Xi\setminus\Xi)\right) \\
&=-\wh{\nu}\left(\mscr{H}(x|g^{-1}x,\mf{h})\cap\Xi\right)+\wh{\nu}\left(\mscr{H}(g^{-1}x|\mf{h})\cap(g^{-1}\Xi\setminus\Xi)\right),
\end{align*}
since $\mscr{H}(g^{-1}x|x)\cap\Xi=\emptyset$, as $\Xi\cu\s_{\xi}\setminus\s_x$. Note that: 
\[\mscr{H}(x|g^{-1}x)\cap g\Xi=g^{-1}\left(\mscr{H}(gx|x)\cap g^2\Xi \right)\cu g^{-1}\left(\mscr{H}(gx|x)\cap\Xi \right)=\emptyset.\]
Thus $\alpha_{\Xi}(g\mf{h})-\alpha_{\Xi}(\mf{h})$ equals:
\begin{align*}
&-\wh{\nu}\left(\mscr{H}(x|g^{-1}x,\mf{h})\cap(\Xi\setminus g\Xi)\right)+\wh{\nu}\left(\mscr{H}(g^{-1}x|\mf{h})\cap(g^{-1}\Xi\setminus\Xi)\right) \\
&\geq-\wh{\nu}\left(\s_x^*\cap(\Xi\setminus g\Xi)\right)+\wh{\nu}\left(\mscr{H}(x|g\mf{h})\cap(\Xi\setminus g\Xi)\right) \\
&\geq-\wh{\nu}\left(\left(\s_x^*\setminus\mscr{H}(x|g\mf{h})\right)\cap(\Xi\setminus g\Xi)\right)\geq-\wh{\nu}\left(\left\{\mf{k}\in\Xi\setminus g\Xi\mid g\mf{h}\not\cu\mf{k}\right\}\right).
\end{align*}
Now, consider for each $g\in K$ the set ${\Om_2(g):=\bigcap_{-r+1\leq i\leq r!}g^{i-1}\Om}$ as in Corollary~\ref{gXi_2}; we have $g\Om_2(g)\cu\Om_2(g)$ and, by part~1 of Lemma~\ref{C_1}, there exists a constant $C_1$ such that each $\mf{h}\in\Om$ with $d(x,\mf{h})>C_1$ lies in $g\Om_2(g)$ for every $g\in K$. By part~1 of Lemma~\ref{strongly red}, if $d(x,g\mf{h})>C_3(\Om,C_1)$ we have
\[\alpha_{\Om_2(g)}(g\mf{h})-\alpha_{\Om_2(g)}(\mf{h})\geq-\wh{\nu}\left(\left\{\mf{k}\in\Om_2(g)\setminus g\Om_2(g)\mid g\mf{h}\not\cu\mf{k}\right\}\right)=0.\]
By part~2 of Lemma~\ref{strongly red}, if $d(x,\mf{h})>C_3(\Om,C_1)$ and $d(x,g\mf{h})>C_3(\Om,C_1)$, we have
\[\alpha_{\Om}(g\mf{h})-\alpha_{\Om}(\mf{h})=\alpha_{\Om_2(g)}(g\mf{h})-\alpha_{\Om_2(g)}(\mf{h})\geq 0. \]
We conclude that $\alpha_{\Om}(g\mf{h})\geq\alpha_{\Om}(\mf{h})$ whenever $d(x,\mf{h})>C_4:=C_3(\Om,C_1)+M$, where $M$ is the maximum of the distances $d(x,gx)$ for $g\in K$.

We now prove part~2. Let $C_4$ be as in part~1. Part~1 of Lemma~\ref{C_1} provides a constant $C_1'$ such that each $\mf{h}\in\Om$ with $d(x,\mf{h})>C_1'$ lies in $g\Om\cap g^{-1}\Om$ for every $g\in K$. Let $C_5$ be the supremum of the values $\alpha_{\Om}(\mf{h})$ for $\mf{h}\in\Om$ with $d(x,\mf{h})\leq\max\{C_4,C_1'\}$; by Lemma~\ref{chains in small UBS's} we have $C_5<\wh{\nu}(\Om)$. If $c>C_5$, any $\mf{h}\in\Om\setminus\Om_c$ satisfies $d(x,\mf{h})>\max\{C_4,C_1'\}$, hence $\alpha_{\Om}(g\mf{h})\geq\alpha_{\Om}(\mf{h})>c$ and $g\mf{h}\in\Om$; in particular $g(\Om\setminus\Om_c)\cu\Om\setminus\Om_c$. 

Observe that $(\Om\setminus g\Om)\cap g\Om_c\cu g(\Om_c\setminus\Om)=\emptyset$ and $\Om\setminus g\Om\cu\Om_c$, by our choice of the constant $C_1'$; thus, $\Om\setminus g\Om\cu \Om_c\setminus g\Om_c$. Conversely, it is clear that $\Om_c\setminus g\Om_c\cu\Om$ and 
\[\left(\Om_c\setminus g\Om_c\right)\cap g\Om=\Om_c\cap g\left(\Om\setminus\Om_c\right)\cu\Om_c\cap\left(\Om\setminus\Om_c\right)=\emptyset .\]
\end{proof}

Consider now, for $c>0$, the functions $F_{c,\Om}:=-\left(1-\frac{\alpha_{\Om}}{c}\right)\mathds{1}_{\Om_c}$ in $L^2\left(\mscr{H},\wh{\nu}\right)$.

\begin{lem}\label{properties of Om_c 4}
Assume that $g\Om\cu\Om$ for all $g\in K$. For every $\eps$, there exists a constant $C_{\eps}=C_{\eps}(\Om,K)<+\infty$ such that $\|(g-\id)F_{c,\Om}-\mathds{1}_{\Om\setminus g\Om}\|_2<\eps$ for all $g\in K$ and all $c\geq C_{\eps}$. If instead $g\Om\supseteq\Om$ for all $g\in K$, we have $\|(g-\id)F_{c,\Om}+\mathds{1}_{g\Om\setminus\Om}\|_2<\eps$ for $c\geq C_{\eps}$.
\end{lem}
\begin{proof}
Observe that:
\begin{align*}
(g-\id)&F_{c,\Om}=-\left(1-\frac{g\alpha_{\Om}}{c}\right)\mathds{1}_{g\Om_c}+\left(1-\frac{\alpha_{\Om}}{c}\right)\mathds{1}_{\Om_c} \\
&=-\left(1-\frac{g\alpha_{\Om}}{c}\right)\mathds{1}_{g\Om_c\setminus\Om_c}+\left(1-\frac{\alpha_{\Om}}{c}\right)\mathds{1}_{\Om_c\setminus g\Om_c}+\frac{g\alpha_{\Om}-\alpha_{\Om}}{c}\mathds{1}_{g\Om_c\cap\Om_c}.
\end{align*}
We will analyse the three summands separately. By Lemma~\ref{properties of Om_c 3}, we have 
\[-\frac{2C_2}{c}\leq\left(1-\frac{\alpha_{\Om}(\mf{h})+C_2}{c}\right)\leq\left(1-\frac{g\alpha_{\Om}(\mf{h})}{c}\right)\leq\left(1-\frac{\alpha_{\Om}(\mf{h})-C_2}{c}\right)<\frac{C_2}{c},\]
for each $\mf{h}\in g\Om_c\setminus\Om_c\cu\Om_{c+C_2}\setminus\Om_c$. By part~2 of Lemma~\ref{properties of Om_c},
\[\left\|\left(1-\frac{g\alpha_{\Om}}{c}\right)\mathds{1}_{g\Om_c\setminus\Om_c}\right\|_2\leq\frac{2C_2}{c}\cdot\wh{\nu}(g\Om_c\setminus\Om_c)^{1/2}\leq\frac{2C_2(r(c+C_2))^{1/2}}{c}\xrightarrow[c\ra+\infty]{} 0.\]
By part~1 of Lemma~\ref{C_1}, there exists a constant $C_1$ such that each $\mf{h}\in\Om$ with $d(x,\mf{h})>C_1$ lies in $g\Om$ for every $g\in K$; in particular, if $\mf{k}\in\Om\setminus g\Om$ for some $g\in K$, then $\alpha_{\Om}(\mf{k})\leq d(x,\mf{k})\leq C_1$. By part~2 of Lemma~\ref{properties of Om_c 2}, if $c\geq C_5$ we have $\Om_c\setminus g\Om_c=\Om\setminus g\Om$ and
\begin{align*}
&\left\|\left(1-\frac{\alpha_{\Om}}{c}\right)\mathds{1}_{\Om_c\setminus g\Om_c}-\mathds{1}_{\Om\setminus g\Om}\right\|_2=\left\|\left(\mathds{1}_{\Om_c\setminus g\Om_c}-\mathds{1}_{\Om\setminus g\Om}\right) - \frac{\alpha_{\Om}}{c}\mathds{1}_{\Om_c\setminus g\Om_c}\right\|_2 \\
&=\left\|\frac{\alpha_{\Om}}{c}\mathds{1}_{\Om\setminus g\Om}\right\|_2\leq\frac{C_1}{c}\cdot\wh{\nu}\left(\Om\setminus g\Om\right)^{1/2} \leq \frac{C_1M^{1/2}}{c}\xrightarrow[c\ra+\infty]{} 0 ,
\end{align*}
where $M$ is the maximum of $\chi_{\Om}(g)$ for $g\in K$, which exists by Proposition~\ref{chi_xi continuous}. Finally, by part~2 of Lemma~\ref{properties of Om_c} and part~1 of Lemma~\ref{properties of Om_c 3},
\[\left\|\frac{g\alpha_{\Om}-\alpha_{\Om}}{c}\mathds{1}_{g\Om_c\cap\Om_c}\right\|_2\leq\frac{C_2}{c}\cdot\wh{\nu}(\Om_c)^{1/2}\leq\frac{C_2(cr)^{1/2}}{c}\xrightarrow[c\ra+\infty]{} 0.\]
If instead $g\Om\supseteq\Om$ for all $g\in K$, the previous discussion shows that, for large $c$, we have $\|(g^{-1}-\id)F_{c,\Om}-\mathds{1}_{\Om\setminus g^{-1}\Om}\|_2<\eps$; the conclusion follows by applying $g$.
\end{proof}

\begin{lem}\label{saving the day}
Let $\Xi\cu\s_{\xi}\setminus\s_x$ be a UBS and let $\Xi^1,...,\Xi^k\cu\Xi$ be pairwise inequivalent, reduced UBS's representing all minimal equivalence classes of UBS's almost contained in $\Xi$; for every $i$, set $\s_i:=\wh{\nu}(\Xi^i)$. There exist increasing sequences $(c_n^{(i)})_{n\geq 0}$ such that 
\begin{itemize}
\item $c_n^{(i)}\ra\s_i$ for all $i=1,...,k$;
\item $\left(\Xi^1\setminus\Xi^1_{c_n^{(1)}}\right)\cup ... \cup\left(\Xi^k\setminus\Xi^k_{c_n^{(k)}}\right)$ is inseparable for all $n\geq 0$.
\end{itemize}
\end{lem}
\begin{proof}
We proceed by induction on $k$. If $k=1$, the lemma is immediate. Suppose that $k\geq 2$; without loss of generality, we can assume that $\Xi^k$ corresponds to a vertex with no incoming edges in the full subgraph of $\mc{G}(\xi)$ with vertices $[\Xi^1],...,[\Xi^k]$. Fix $\eps>0$; we will construct $c^{(i)}>\s_i-\eps$ satisfying the inseparability condition.

Pick a diverging chain $(\mf{h}_n^{(i)})_{n\geq 0}$ in each $\Xi^i$; up to replacing $(\mf{h}_n^{(k)})_{n\geq 0}$ with a cofinite subchain, we can assume that, for each $i\leq k-1$ and each $m\geq 0$, the halfspace $\mf{h}_m^{(k)}$ is transverse to $\mf{h}_n^{(i)}$ for almost every $n$. By Lemma~\ref{chains in small UBS's}, there exists $c^{(k)}>\s_k-\eps$ such that $\Xi^k\setminus\Xi^k_{c^{(k)}}$ is contained in the inseparable closure of $\{\mf{h}_n^{(k)}\}_{n\geq 0}$. As a consequence, for every $\mf{h}\in\Xi^k\setminus\Xi^k_{c^{(k)}}$ and every $i\leq k-1$, the halfspaces $\mf{h}$ and $\mf{h}_n^{(i)}$ are transverse for almost every $n$.

Halfspaces lying in the inseparable closure of $\Xi^1\cup ... \cup\Xi^k$, but in neither of the $\Xi^i$, are at uniformly bounded distance from $x$, by part~3 of Proposition~\ref{main prop on UBS's}; say that these distances are bounded above by $M<+\infty$. Enlarge $M$ so that all $\mf{j}\in\Xi^k$ with $d(x,\mf{j})>M$ lie in $\Xi^k\setminus\Xi^k_{c^{(k)}}$. By part~4 of Proposition~\ref{main prop on UBS's} there exists a UBS contained in $\Xi$ such that $[\Xi^1],...,[\Xi^{k-1}]$ are all the equivalence classes of minimal UBS's almost contained in $\Xi$. The inductive hypothesis and Lemma~\ref{chains in small UBS's} imply that we can find $c^{(i)}>\s_i-\eps$ so that $\left(\Xi^1\setminus\Xi^1_{c^{(1)}}\right)\cup ... \cup\left(\Xi^{k-1}\setminus\Xi^{k-1}_{c^{(k-1)}}\right)$ is inseparable and $d(x,\mf{h})>M$ for all $\mf{h}\in\Xi^i\setminus\Xi^i_{c^{(i)}}$ with $i\leq k-1$.

Now, if $\left(\Xi^1\setminus\Xi^1_{c^{(1)}}\right)\cup ... \cup\left(\Xi^k\setminus\Xi^k_{c^{(k)}}\right)$ were not inseparable, there would exist $\mf{j}\in\mscr{H}$ such that $\mf{k}_u\cu\mf{j}\cu\mf{k}_v$, for halfspaces $\mf{k}_u\in\left(\Xi^u\setminus\Xi^u_{c^{(u)}}\right)$ and $\mf{k}_v\in\left(\Xi^v\setminus\Xi^v_{c^{(v)}}\right)$, but $\mf{j}$ would not lie in any of the $\Xi^i\setminus\Xi^i_{c^{(i)}}$. In particular, $u\neq v$ and $k\in\{u,v\}$; observe that $v\neq k$, otherwise $\mf{k}_v\in\Xi^k\setminus\Xi^k_{c^{(k)}}$ would not be transverse to diverging chains in $\Xi^u$. Thus, $u=k$; moreover, for all $i\leq k-1$, the halfspace $\mf{j}$ must be transverse to $\mf{h}_n^{(i)}$ for almost every $n$, since $v\leq k-1$ and $\mf{j}$ does not lie in $\left(\Xi^1\setminus\Xi^1_{c^{(1)}}\right)\cup ... \cup\left(\Xi^{k-1}\setminus\Xi^{k-1}_{c^{(k-1)}}\right)$, which is inseparable. Since $d(x,\mf{j})\geq d(x,\mf{k}_v)>M$, we have $\mf{j}\in\Xi^1\cup ...\cup\Xi^k$; the fact that each $\Xi_i$ is reduced implies that $\mf{j}\in\Xi^k$. By our choice of $M$, we have $\mf{j}\in\Xi^k\setminus\Xi^k_{c^{(k)}}$, a contradiction.
\end{proof}

\begin{lem}\label{UBS's vs points}
There exists a UBS $\Xi_{\xi}\cu\s_{\xi}\setminus\s_x$ such that every UBS $\Xi\cu\Xi_{\xi}$ with $\Xi\sim\Xi_{\xi}$ is of the form $\s_{\xi}\setminus\s_y$ for some $y\in X$, up to a null set.
\end{lem}
\begin{proof}
Let $\Xi_1,...,\Xi_k\cu\s_{\xi}\setminus\s_x$ be pairwise inequivalent minimal UBS's representing all minimal elements of $\overline{\mc{U}}(\xi)$; we can assume that they are all reduced by Lemma~\ref{reduced and strongly reduced}. Halfspaces in $\s_{\xi}\setminus\s_x$ that are transverse to a diverging chain in each $\Xi_i$ lie at uniformly bounded distance from $x$, by part~3 of Proposition~\ref{main prop on UBS's}; say that these distances are bounded above by $M<+\infty$. Let $\Xi_{\xi}$ consist of all $\mf{h}\in\s_{\xi}\setminus\s_x$ with $d(x,\mf{h})>M$; it is a UBS equivalent to $\s_{\xi}\setminus\s_x$. 

Observe that, if $\Xi\cu\Xi_{\xi}$ is a UBS equivalent to $\Xi_{\xi}$ and $\xi\in\wt{\mf{h}}$ for halfspaces ${\mf{h}\cu\mf{k}\in\Xi}$, then $\mf{h}\in\Xi$. Indeed, otherwise $\mf{h}$ would not contain any halfspace of $\Xi$, by inseparability, and it would therefore be transverse to a diverging chain in each of the $\Xi_i$; hence $d(x,\mf{k})\leq d(x,\mf{h})\leq M$, a contradiction. 

Now, given $\Xi\cu\Xi_{\xi}$, consider the set $\s:=(\s_{\xi}\setminus\Xi)\sqcup\Xi^*\cu\mscr{H}$. Observe that $\s$ is an ultrafilter. Indeed, since $\s_{\xi}\setminus\Xi\cu\s_{\xi}$ and $\Xi^*\cu\s_x$, it suffices to check that $\mf{h}\cap\mf{k}\neq\emptyset$ whenever $\mf{h}\in\s_{\xi}\setminus\Xi$ and $\mf{k}\in\Xi^*$. If such halfspaces were disjoint, we would have $\xi\in\wt{\mf{h}}$ and $\mf{h}\cu\mf{k}^*\in\Xi$, contradicting the observation we made above since $\mf{h}\not\in\Xi$.

Finally, by Lemma~4.3 in \cite{Fioravanti2}, we can decompose $\s_{\xi}\setminus\s_x=\Xi\sqcup\Sigma$ for some set $\Sigma$ with $\wh{\nu}(\Sigma)<+\infty$; in particular, $\s_x\setminus\s_{\xi}=\Xi^*\sqcup\Sigma^*$. Since $\Xi$ and $\s_x$ are disjoint, we have $\s_x\setminus\s=\s_x\setminus\left(\s_{\xi}\cup\Xi^*\right)=\Sigma^*$, hence $\wh{\nu}(\s_x\triangle\s)<+\infty$. Proposition~\ref{M(X)=X} implies that there exists $y\in X$ such that $\s\triangle\s_y$ is null; thus, up to measure zero, $\s_{\xi}\setminus\s_y=\s_{\xi}\setminus\s=\Xi$.
\end{proof}

We are now ready to prove Proposition~\ref{point x_S}.

\begin{proof}[Proof of Proposition~\ref{point x_S}.]
Let $\Xi^1, ..., \Xi^k\cu\s_{\xi}\setminus\s_x$ be pairwise inequivalent UBS's representing all minimal elements of the poset $\left(\overline{\mc{U}}(\xi),\preceq\right)$; by Lem\-ma~\ref{reduced and strongly reduced}, we can assume that each $\Xi^i$ is strongly reduced. Up to replacing each $\Xi^i$ with a smaller UBS, part~2 of Lemma~\ref{C_1} guarantees that we can assume that $\Xi^i\cap\Xi^j=\emptyset$ and $\Xi^i\cap g\Xi^j$ for all $i\neq j$ and $g\in K$. By part~2 of Lemma~\ref{properties of Om_c 2}, we have ${g(\Xi^i\setminus\Xi^i_c)\cu\Xi^i\setminus\Xi^i_c}$ for all ${C_5(\Xi^i,K)\leq c<\wh{\nu}(\Xi^i)}$ and all $g\in K$ with $\chi_{\Xi^i}(g)\geq 0$, while we have ${g(\Xi^i\setminus\Xi^i_c)\supseteq\Xi^i\setminus\Xi^i_c}$ if $\chi_{\Xi^i}(g)\leq 0$. Lemma~\ref{saving the day} provides constants ${C_5(\Xi^i,K)\leq c_i<\wh{\nu}(\Xi^i)}$ such that ${\Om_K:=\left(\Xi^1\setminus\Xi^1_{c_1}\right)\cup ... \cup\left(\Xi^k\setminus\Xi^k_{c_k}\right)}$ is inseparable. Thus $\Om_K$ is a UBS equivalent to $\s_{\xi}\setminus\s_x$ by part~3 of Proposition~\ref{main prop on UBS's}. We conclude by Lem\-ma~\ref{UBS's vs points}, enlarging the constants $c_i$ if necessary, so that $\Om_K\cu\Xi_{\xi}$; this is possible by Lemma~\ref{chains in small UBS's}.
\end{proof}

\bibliography{mybib}
\bibliographystyle{alpha}

\end{document}